\documentclass[11pt,reqno]{amsart}
\usepackage{geometry}                		
\usepackage{graphicx}				
\usepackage{amssymb}
\usepackage{amsmath,amssymb,mathrsfs,amsfonts,amsthm,mathptmx}
\usepackage{graphicx,cite,times,color,bm,dcolumn}
\numberwithin{equation}{section}
\numberwithin{figure}{section}
\numberwithin{table}{section}
\newtheorem{theorem}{Theorem}[section]
\newtheorem{lemma}{Lemma}[section]
\newtheorem{definition}{Definition}[section]

\newtheorem{proposition}{Proposition}[section]
\newtheorem{remark}{Remark}[section]
\newtheorem{assumption}{Assumption}[section]

\allowdisplaybreaks[4]
\begin{document}
\title[Symplectic dG full discretization]{A  symplectic discontinuous Galerkin full discretization for stochastic Maxwell equations}\thanks{This work is funded by National Natural Science Foundation of China (No. 11871068, No. 12022118, No. 11971470 and No. 12031020).}
\author{Chuchu Chen}
\address{LSEC, ICMSEC, Academy of Mathematics and Systems Science, Chinese Academy of Sciences, and School of Mathematical Sciences, University of Chinese Academy of Sciences, Beijing 100049, China}
\email{chenchuchu@lsec.cc.ac.cn}
	\begin{abstract}
	This paper proposes a fully discrete method called the  symplectic dG full discretization for stochastic Maxwell equations driven by additive noises, based on a stochastic symplectic method in time and a discontinuous Galerkin (dG) method with the upwind fluxes in space. A priori  $H^k$-regularity ($k\in\{1,2\}$) estimates for the solution of stochastic Maxwell equations are presented, which  have not been reported before to the best of our knowledge. These $H^k$-regularities are vital to make  the assumptions of the mean-square convergence analysis  on the initial fields, the noise and the medium coefficients, but not on the solution itself. The convergence order of the symplectic dG full discretization is shown to be $k/2$ in the temporal direction and $k-1/2$ in the spatial direction. Meanwhile we  reveal the small noise asymptotic  behaviors of the exact and numerical solutions via the large deviation principle, and show that the fully discrete method preserves the divergence relations in a weak sense.
	\end{abstract}

\maketitle

\section{Introduction}\label{sec:1}
Stochastic Maxwell equations are often used to better understand the role of thermodynamic fluctuations presented in the electromagnetic fields, and to get a deeper insight regarding the propagation of electromagnetic waves in complex media (see e.g. \cite{RKT1989}). A mathematically rigorous framework on the effects of randomness  has been developed in \cite{RSY2012}. 
The numerical treatment of the three dimensional stochastic Maxwell equations, even in the linear case, is a challenging task, due to the interaction of the large scale  and the randomness of the problem.
In this paper, we first discretize stochastic Maxwell equations in time via the midpoint scheme, which inherits the stochastic symplecticity of the original continuous problem, and subsequently  in space based on a   dG method combining its attractive features on the treatment of complex geometries and  composite media. 

For the time-dependent stochastic Maxwell equations, there exist some works on the construction of full discretizations, for example, multi-symplectic numerical methods (cf. \cite{CHZ2016, HJZ2014}), energy-conserving methods (cf. \cite{HJZC2017}).  On the rigorous error analysis of the numerical approximations, the existing works mainly focus on  the temporal semidiscretizations (see \cite{CHJ2019a, CHJ2019b, CCHS2020}). It is shown in \cite{CHJ2019a} that a semi-implicit Euler scheme  converges with order $1/2$ in mean-square sense, and in \cite{CCHS2020} that the exponential integrators have mean-square convergence order $1/2$, when applied to stochastic Maxwell equation with multiplicative It\^o noise.  Authors in \cite{CHJ2019b} show that the stochastic symplectic Runge-Kutta semidiscretizations are mean-square convergent with order  $1$ in the additive case.
As far as we know, there are few works on the rigorous error analysis of the spatio-temporal full discretizations for the time-dependent stochastic Maxwell equations. The difficulty lies in  the lack of regularity of the solution in $H^k$-norms or even in $C^{k}$-norms, which depends on the spatial domain, the medium coefficients and the noise, etc. For example, on a cuboid, the solution of the  time-harmonic deterministic Maxwell equations only has $H^{\alpha}$-regularity for $\alpha<3$ in general.

In this work, we consider the approximation of the stochastic electric and magnetic fileds ${\bf E}(t,x)$ and ${\bf H}(t,x)$  satisfying the following stochastic Maxwell equations on a cuboid  $D=(a_1^-,a_1^+)\times (a_2^-,a_2^+)\times (a_3^-,a_3^+)\subset {\mathbb R}^{3}$,
\begin{subequations}\label{sto_max}
\begin{align}
&\varepsilon {\rm d}{\bf E}-\nabla\times {\bf H}{\rm d}t=-{\rm d}W_e(t),\qquad(t,{\bf x})\in(0,~T]\times D,\label{sto_max1} \\
&\mu {\rm d}{\bf H}+\nabla\times {\bf E}{\rm d}t=-{\rm d}W_m(t),\qquad(t,{\bf x})\in(0,~T]\times D,\label{sto_max2}\\
&\nabla\cdot(\varepsilon {\bf E})=0,~\nabla\cdot(\mu {\bf H})=0,\qquad(t,{\bf x})\in(0,~T]\times D,\label{sto_max3}\\
&{\bf n}\times {\bf E}={\bf 0},~{\bf n}\cdot (\mu {\bf H})=0,\qquad(t,{\bf x})\in(0,~T]\times\partial D,\label{sto_max4}\\
&{\bf E}(0,{\bf x})={\bf E}_0({\bf x}),~{\bf H}(0,{\bf x})={\bf H}_0({\bf x}),\qquad{\bf x}\in D,\label{sto_max5}
\end{align}
\end{subequations}
where $T>0$, and ${\bf n}({\bf x})$ denotes the outer unit normal at ${\bf x}\in \partial D$.
We suppose that the medium is isotropic, which implies that the permittivity $\varepsilon$ and the permeability $\mu$  are real-valued scalar functions, i.e., $\varepsilon, \mu: D \rightarrow  {\mathbb R}$.
Throughout this paper, we assume the medium coefficients satisfy 
\begin{equation}\label{assump_coe}
\varepsilon,\; \mu \in L^{\infty}(D),~~
\varepsilon, \; \mu \geq \delta \mbox{ for a constant } \delta>0.
\end{equation}
Here $W_e(t)$ (resp. $W_m(t)$) is a $Q_e$-Wiener (resp. $Q_m$-Wiener) process  with respect to a filtered probability space $(\Omega,{\mathcal F},\{{\mathcal F}_{t}\}_{0\leq t\leq T},{\mathbb P})$ with $Q_e$ (resp. $Q_m$) being a  symmetric, positive definite operator with finite trace on $U=L^2(D)^3$. Moreover, $W_e(t)$ and $W_m(t)$ are independent. The phase flow of \eqref{sto_max} preserves the stochastic symplecticity (cf. \cite{CHJ2019b}), i.e., if $\varepsilon,\mu$ are constants, for any $t\in[0,T]$,
$
\overline{\omega}(t)=\int_{D}d{\bf E}(t)\wedge d{\bf H}(t)=\overline{\omega}(0)$, ${\mathbb P}\mbox{-a.s.}$

The solution theory of \eqref{sto_max}, which is crucial in the mean-square error analysis, is presented in Section \ref{sec:2} with certain assumptions being made on the medium coefficients, the initial fields and the noise. We restrict the Maxwell operator $M$ on the closed subspace ${\mathbb V}_0$  of ${\mathbb V}:=L^2(D)^3\times L^2(D)^3$, in order to respect to all boundary conditions and divergence properties.
 These conditions and properties are important to get the $L^{p}(\Omega; C([0,T];H^1(D)^6))$-regularity ($H^1$-regularity in short) for the solution of \eqref{sto_max}, under the first order regularity and certain compatibility conditions of the initial data and the noise term; see Proposition \ref{exact_H1}. Furthermore, we can guarantee that the solution has $H^2$-regularity if more assumptions on the medium coefficients, the initial fields and the noise are employed; see Proposition \ref{exact_H2}.

In order to inherit the stochastic symplectic structure, we apply the midpoint scheme \eqref{midpoint} to discretize  \eqref{sto_max} in time in Section \ref{sec:3}.
The error is measured in $L^2(\Omega;{\mathbb V})$, and gives a bound of order $k/2$ provided that the solutions of the continuous problem \eqref{sto_max} and the temporal semidiscretization \eqref{midpoint} belong to ${\mathcal D}(M^k)$ with $k\in\{1,2\}$.
It is also shown that the divergence conservation laws \eqref{sto_max3} are preserved numerically by the semidiscretization \eqref{midpoint} in time.

We discretize the temporal semidiscretization \eqref{midpoint} further in space using a dG method, and then it results the fully discrete method \eqref{full discretization}, called the symplectic dG full discretization; see also Section \ref{sec4} for the treatment of the dG approximation of stochastic Maxwell equations. 
We refer interested readers  to \cite{Zhang2008} for  the application of dG methods to the time-harmonic stochastic Maxwell equations with color noise, 
to \cite{CZZ2008} for the application to stochastic Helmholtz-type equation, to 
\cite{A2020} for the application to stochastic Allen-Cahn equation, to \cite{BLM2020} for the application to the semi-linear stochastic wave equation, to \cite{LST2020} for the application to  stochastic conservation laws, and to \cite{CHJ2017} for the application of a symplectic local dG method to stochastic Schr\"odinger equation.
Since the highest regularity of stochastic Maxwell equations that can be guaranteed is in $H^2$, the dG space is taken to be the set of piecewise linear functions. The upwind fluxes are utilized, due to the higher convergence order than the central fluxes; see \cite{HP2015} for the deterministic case. It is shown in Theorem \ref{thm: error of eh} that the mean-square convergence order of the dG approximation \eqref{dG} is of $k-1/2$ if the exact solution of \eqref{sto_max} belongs to $L^{p}(\Omega; C([0,T];H^k(D)^6))$ with $k\in\{1,2\}$. This convergence analysis is presented in a form applied also to the full discretization  \eqref{full discretization}, which is stated in Section \ref{sec:5}. We also show that the divergence properties \eqref{sto_max3} are preserved numerically in a weak sense by the spatial semidiscretization \eqref{dG} and the full discretization \eqref{full discretization} in Proposition \ref{prop: divergence free} and Proposition \ref{prop: divergence free full}, respectively. Moreover, the asymptotic behaviors of the exact and numerical solutions of stochastic Maxwell equations with small noise are investigated in Sections \ref{sec:2}-\ref{sec:5}, respectively. 

To conclude, the main contribution of this paper is to provide a rigorous error analysis of a full discretization for stochastic Maxwell equations.  In particular, we prove that:
\begin{itemize}
\item[(i)] the exact solution and the numerical solution of temporal semidiscrete method belong to $L^{p}(\Omega; C([0,T];H^k(D)^6))$ with $k\in\{1,2\}$ depending only on the assumptions on the the medium coefficients, the initial fields and the noise, which have not been reported before to the best of our knowledge;
\item[(ii)] the mean-square error of the full discretization in $L^2(\Omega;{\mathbb V})$ is of order $k/2$ in time and of order $k-1/2$ in space $(k\in\{1,2\})$, which retains the convergence order of the upwind fluxes space discretization in the deterministic case.

\end{itemize} 

\section{Properties of stochastic Maxwell equations} \label{sec:2}
This section presents the notations and basic results for stochastic Maxwell equations, including the stochastic symplectic structure, the regularity in $L^{p}(\Omega; C([0,T];H^k(D)^6))$ with $k\in\{1,2\}$, and the small noise asymptotic behavior. Throughout this paper, we use $C$ to denote a generic constant, independent of the step sizes $\tau$ and $h$, which may differ from line to line. Let $\Gamma_j^{\pm}$ be the open faces of $D$ given by $x_j=a_j^{\pm}$, respectively, for $j=1,2,3$.
\subsection{Preliminaries}
We first collect notations  used throughout this paper.
We use the standard Sobolev spaces $W^{k,p}(D):=W^{k,p}(D,{\mathbb R})$ for $k\in {\mathbb N}_0$, $p\in[1,\infty]$, where we denote $H^{k}(D)=W^{k,2}(D)$. 
For a real number $\gamma\in(0,1)$ and a normed real vector space $V$, 
denote $C^{\gamma}([0,T];V):=\{f: [0,T]\to V~\mbox{with}~ \|f\|_{C^{\gamma}([0,T];V)}<\infty\}$ the space of all $\gamma$-H\"older continuous functions from $[0,T]$ to $V$, where
\[
\|f\|_{C^{\gamma}([0,T];V)}:=\sup_{t\in[0,T]}\|f(t)\|_{V}+\sup_{t_1,t_2\in[0,T], t_1\neq t_2}\frac{\|f(t_2)-f(t_1)\|_{V}}{|t_2-t_1|^{\gamma}}
\]

Stochastic Maxwell equations \eqref{sto_max} are studied in the  the real Hilbert space ${\mathbb V}=L^2(D)^3\times L^2(D)^3$, endowed with the inner product
 	$
 	\left\langle \begin{pmatrix}
 	{\bf E}_1\\{\bf H}_1
 	\end{pmatrix},~ \begin{pmatrix}
 	{\bf E}_2\\{\bf H}_2
 	\end{pmatrix}\right\rangle_{\mathbb V}=\int_{D}(\varepsilon {\bf E}_1\cdot {\bf E}_2
 	+\mu{\bf H}_1\cdot{\bf H}_2){\rm d}{\bf x}
 	$
for all $({\bf E}_1^{\top}, {\bf H}_1^{\top})^{\top},\; ({\bf E}_2^{\top},{\bf H}_2^{\top})^{\top}\in {\mathbb V}$, and the norm
 	$
 	\left\|\begin{pmatrix}
 	{\bf E}\\{\bf H}
 	\end{pmatrix}\right\|_{\mathbb V}=\left[\int_{D}\left(\varepsilon|{\bf E}|^2+\mu|{\bf H}|^2\right){\rm d}{\bf x}\right]^{1/2},\quad \forall~({\bf E}^{\top}, {\bf H}^{\top})^{\top}\in  {\mathbb V}.$
	This space ${\mathbb V}$ is equivalent to the usual $L^2(D)^6$ space under the assumption \eqref{assump_coe} on the coefficients $\varepsilon$ and $\mu$.
	
	In addition we use the Hilbert spaces
 \begin{align*}
     H({\rm curl},D)&:=\{ v\in L^2(D)^3:~\nabla\times v\in L^2(D)^3 \},\\[1mm]
     H_0({\rm curl},D)&:=\{ v\in H({\rm curl},D):~{\bf n}\times v|_{\partial D}={\bf 0} \},
 \end{align*}
 endowed with the norm
 \begin{equation*}
 \|u\|_{\rm curl}^2=\|u\|_{L^2(D)^3}^2+\|\nabla\times u\|^2_{L^2(D)^3},
 \end{equation*}
 and 
  \begin{align*}
     H({\rm div},D)&:=\{ v\in L^2(D)^3:~\nabla\cdot v\in L^2(D) \},\\[1mm]
     H_0({\rm div},D)&:=\{ v\in H({\rm div},D):~{\bf n}\cdot v|_{\partial D}={\bf 0} \},
 \end{align*}
 endowed with the norm
 \begin{equation*}
  \|u\|_{\rm div}^2=\|u\|_{L^2(D)^3}^2+\|\nabla\cdot u\|^2_{L^2(D)}.
 \end{equation*}
 
 
 After these preparations we introduce the Maxwell operator
 	\begin{equation}\label{M_operator}
 	M=\begin{pmatrix}
 	0& \varepsilon^{-1}\nabla\times \\
 	-\mu^{-1}\nabla\times &0 \\
 	\end{pmatrix},
	\quad 	{\mathcal D}(M)=H_0({\rm curl},D)\times H({\rm curl},D)
 	\end{equation}
on ${\mathbb V}$. 
By defining $u(t)=({\bf E}(t)^{\top},{\bf H}(t)^{\top})^{\top}$, the system \eqref{sto_max} can be rewritten as a stochastic evolution equation
\begin{equation}\label{sto_evo}
\begin{cases}
{\rm d}u(t)=Mu(t) {\rm d}t -{\rm d}W(t),\\
u(0)=u_0,
\end{cases}
\end{equation}
 where $W(t)=(\varepsilon^{-1} W_e(t)^{\top}, \; \mu^{-1}W_m(t)^{\top})^{\top}$ is a $Q$-Wiener process on ${\mathbb V}$ with 
 \[Q=\begin{pmatrix}
 \varepsilon^{-1}Q_e & 0\\
 0 & \mu^{-1} Q_m
 \end{pmatrix}.\]
 In fact, for any $a=(a_1^{\top},a_2^{\top})^{\top}$, $b=(b_1^{\top},b_2^{\top})^{\top}\in{\mathbb V}$, we have
 \begin{align*}
& {\mathbb E}\left[ \langle W(t), a \rangle_{\mathbb V} \langle W(t), b \rangle_{\mathbb V}\right]\\
 &={\mathbb E}\left[ \left(\langle W_e(t), a_1 \rangle_{U}+\langle W_m(t), a_2 \rangle_{U}\right) \left(\langle W_e(t), b_1 \rangle_{U}+\langle W_m(t), b_2 \rangle_{U}\right)\right]\\
 &=\langle Q_e a_1, b_1 \rangle_{U}+\langle Q_m a_2, b_2 \rangle_{U}
 =\langle Q a, b\rangle_{\mathbb V}.
 \end{align*}
 Note that 
$
 {\mathbb E}\|W(t)\|_{\mathbb V}^2
 =t\big(\|\varepsilon^{-\frac12}Q_{e}^{\frac12}\|^2_{HS(U,U)}+
 \|\mu^{-\frac12}Q_{m}^{\frac12}\|^2_{HS(U,U)}\big),
$
 and $Q$ still is a symmetric, positive definite operator on ${\mathbb V}$ with trace  ${\rm Tr}(Q)=\big(\|\varepsilon^{-\frac12}Q_{e}^{\frac12}\|^2_{HS(U,U)}+
 \|\mu^{-\frac12}Q_{m}^{\frac12}\|^2_{HS(U,U)}\big)$.
 It is not difficult to show that the energy of the system \eqref{sto_max} evolutes linearly with a rate ${\rm Tr}(Q)$, i.e., 
$ 
 {\mathbb E}\|u(t)\|_{\mathbb V}^2= {\mathbb E}\|u_0\|_{\mathbb V}^2 +{\rm Tr}(Q) t.
 $

Note that \eqref{sto_evo} is an infinite-dimensional Hamiltonian system. If the coefficients $\varepsilon,\mu$ are constants, 
the canonical form of the infinite-dimensional Hamiltonian system of  \eqref{sto_evo} 
reads
\begin{equation}\label{infinite}
{\rm d}u(t)={\mathbb J}^{-1}\frac{\delta {\mathcal H}}{\delta u}{\rm d}t +{\mathbb J}^{-1}\frac{\delta {\mathcal H_1}}{\delta u}{\rm d}\tilde{W}_e+{\mathbb J}^{-1}\frac{\delta {\mathcal H_2}}{\delta u}{\rm d}\tilde{W}_m,
\end{equation}
where
${\mathbb J}=\begin{pmatrix} 0 & I_3 \\-I_3 &0 \end{pmatrix}$ with $I_3$ being the identity matrix on ${\mathbb R}^{3\times 3}$, $\tilde{W}_e=({\bf 0}^{\top},W_e^{\top})^{\top}$, $\tilde{W}_m=(W_m^{\top},{\bf 0}^{\top})^{\top}$, and 
 ${\mathcal H}=-\frac12\int_D \left( \mu^{-1}{\bf E}\cdot(\nabla\times{\bf E})+\varepsilon^{-1}{\bf H}\cdot(\nabla\times{\bf H}) \right){\rm d}{\bf x},$
${\mathcal H}_1=\int_D\varepsilon^{-1}{\bf H}{\rm d}{\bf x}, ~~{\mathcal H}_2=-\int_D\mu^{-1}{\bf E}{\rm d}{\bf x}.$ The phase flow of \eqref{infinite} preserves the stochastic symplecticity, i.e., for any $t\in[0,T]$,
$
\overline{\omega}(t)=\int_{D}d{\bf E}\wedge d{\bf H}{\rm d}{\bf x},~ {\mathbb P}\mbox{-a.s.}
$
We refer to \cite{CHJ2019b} for the discussion on the symplecticity of stochastic Maxwell equations and the numerical preservation of the symplecticity by the semidiscrete methods in time.

 The domain ${\mathcal D}(M)$ includes the electric boundary condition, but neither the magnetic boundary condition nor the divergence conditions. In order to regard all conditions, we define ${\mathbb V}_0:=\{({\bf E}^{\top},{\bf H}^{\top})^{\top}\in {\mathbb V}:~ \nabla\cdot(\varepsilon {\bf E})=\nabla \cdot(\mu{\bf H})=0,\; {\bf n}\cdot (\mu{\bf H})=0 \mbox{ on } {\partial D} \}$, which is a closed subspace of ${\mathbb V}$ with the inner product and norm being defined the same as in ${\mathbb V}$.  We mainly work with the restriction $M_0$ of $M$ on ${\mathbb V}_0$.
It is known that under \eqref{assump_coe}, $M_0: {\mathcal D}(M_0)={\mathcal D}(M)\cap {\mathbb V}_0\to {\mathbb V}_0$ is skew adjoint, and thus generates a unitary $C_0$-group $\{S(t)\}_{t\in{\mathbb R}}$ on ${\mathbb V}_0$. Moreover, since $M$ maps ${\mathcal D}(M)$ into ${\mathbb V}_0$, we have ${\mathcal D}(M_0^{k})={\mathcal D}(M^k)\cap {\mathbb V_0}$ (cf. \cite{HJS2015}).

 \subsection{$H^1$-regularity}
The $H^1$-regularity of the solution is deduced by utilizing the fact that $v\in H({\rm curl},D)\cap H({\rm div},D)$ belongs to $H^1(D)^3$ if $v\times {\bf n}={\bf 0}$ or $v\cdot{\bf n}=0$ holds on $\partial D$. Moreover, the $H^1$-norm of $v$ is dominated by
 $
\|v\|_{H^1(D)^3}\leq C\left( \|v\|_{L^2(D)^3}+\|\nabla\times v\|_{L^2(D)^3}
+\|\nabla\cdot v\|_{L^2(D)}\right),
$
where the constant $C$ depends on the space domain $D$.
Since $\nabla\cdot(\varepsilon {\bf E})=0$, we get that
$
\nabla\cdot{\bf E}=\nabla\cdot\big(\varepsilon^{-1}\varepsilon {\bf E}\big)
=\varepsilon^{-1}\nabla\cdot(\varepsilon{\bf E})+\nabla(\varepsilon^{-1})\cdot (\varepsilon{\bf E})=-\varepsilon^{-1}\nabla\varepsilon\cdot{\bf E}
$
belongs to $L^2(D)^3$ if $\varepsilon\in W^{1,\infty}(D)$ with $\varepsilon\geq\delta>0$ for a constant $\delta>0$, and analogously for ${\bf H}$. 
That means that $\|\nabla\cdot{\bf E}\|_{L^2(D)}+\|\nabla\cdot{\bf H}\|_{L^2(D)}\leq C\big(\delta, \|\varepsilon\|_{W^{1,\infty}(D)}, \|\mu\|_{W^{1,\infty}(D)}\big)\|({\bf E},{\bf H})\|_{L^2(D)^6}$.
Hence, ${\mathcal D}(M_0)={\mathcal D}(M)\cap {\mathbb V}_0
\hookrightarrow H^1(D)^6$, if coefficients $\varepsilon,\mu$ satisfies certain assumptions as above. Moreover, 
\begin{equation}\label{embed H1}
\|({\bf E},{\bf H})\|_{H^1(D)^6}\leq C \|({\bf E},{\bf H})\|_{{\mathcal D}(M_0)},
\end{equation}
with $C:=C\big(\delta, \|\varepsilon\|_{W^{1,\infty}(D)}, \|\mu\|_{W^{1,\infty}(D)}\big)$.
\begin{proposition}\label{exact_H1}
Let the assumption \eqref{assump_coe} hold, and let $Q^{\frac12}\in HS({\mathbb V},\;{\mathcal D}(M_0))$ and $u_0\in L^p(\Omega;{\mathcal D}(M_0))$ for some $p\geq 2$. Then the equation \eqref{sto_evo} has a unique solution $u\in L^{p}(\Omega;\; C([0,T]; {\mathcal D}(M_0)))$ given by
 \begin{equation}\label{mild solution}
 u(t)=S(t)u_0-\int_0^tS(t-s){\rm d}W(s),
\end{equation}
where $u$ also belongs to $C^{\frac12}([0,T]; \; L^p(\Omega;\; {\mathbb V}_0))$.
Assume further that $\varepsilon,\,\mu\in W^{1,\infty}(D)$,
then
\begin{equation}\label{estimate H1}
{\mathbb E}\bigg[\sup_{t\in[0,T]}\|u(t)\|_{H^1(D)^6}^{p}\bigg]\leq C{\mathbb E}\bigg[\sup_{t\in[0,T]}\|u(t)\|_{{\mathcal D}(M_0)}^{p}\bigg]
\leq C(1+{\mathbb E}\|u_0\|_{{\mathcal D}(M_0)}^{p}),
\end{equation}
where  $C$ depends on  $T$, $\delta$, $\|\varepsilon\|_{W^{1,\infty}(D)}$, $\|\mu\|_{W^{1,\infty}(D)}$ and $\|Q^{\frac12}\|_{HS({\mathbb V},\;{\mathcal D}(M_0))}$.
\end{proposition}
\begin{proof}
Since $M_0$ generates a unitary $C_0$-group $\{S(t)\}_{t\in{\mathbb R}}$ on ${\mathbb V}_0$, the existence and uniqueness of the mild solution $u(t)$ of \eqref{mild solution} on ${\mathbb V}_0$ follows.
 The estimate on stochastic convolution yields
\begin{align}\label{eq 2.61}
&\Big[{\mathbb E}\big(\sup_{t\in[0,T]} \|u(t)\|_{{\mathcal D}(M_0)}^{p}\big)\Big]^{\frac{1}{p}}
\leq \Big[{\mathbb E}\big(\|u_0\|_{{\mathcal D}(M_0)}^{p}\big)\Big]^{\frac1p}\nonumber\\
&~+\bigg[{\mathbb E}\bigg(\sup_{t\in[0,T]}\Big\|\int_0^tS(t-s){\rm d}W(s)\Big\|_{{\mathcal D}(M_0)}^{p}\bigg) \bigg]^{\frac1p}
\leq  C\Big(1+\Big[{\mathbb E}\big(\|u_0\|_{{\mathcal D}(M_0)}^{p}\big)\Big]^{\frac1p}\Big),
\end{align}
where the constant $C$ depends on $T$ and $\|Q^{\frac12}\|_{HS({\mathbb V},\;{\mathcal D}(M_0))}$.

Based on \cite[Lemma 3.3]{CHJ2019a} and \eqref{eq 2.61}, for any $0\leq s\leq t \leq T$, we get
\begin{align*}
\|u(t)-u(s)\|_{L^p(\Omega; {\mathbb V}_0)} &\leq \|\big(S(t-s)-I\big)u(s)\|_{L^p(\Omega; {\mathbb V}_0)}+\left\|\int_s^t S(t-r){\rm d}W(r)\right\|_{L^p(\Omega; {\mathbb V}_0)}\\
&\leq C(1+\|u_0\|_{L^p(\Omega; {\mathcal D}(M))})(t-s)+C(t-s)^{\frac12},
\end{align*}
which leads to
\begin{align*}
\|u\|_{C^{\frac12}([0,T]; \; L^p(\Omega;\; {\mathbb V}_0))}
=\sup_{t\in[0,T]}\|u(t)\|_{L^p(\Omega;\; {\mathbb V}_0)}+\sup_{t\neq s}
\frac{\|u(t)-u(s)\|_{L^p(\Omega;\; {\mathbb V}_0)}}{|t-s|^{\frac12}}
\leq C.
\end{align*}

Utilizing the embedding \eqref{embed H1}, the assertion \eqref{estimate H1} follows from \eqref{eq 2.61}. Thus the proof is finished.
\end{proof}
 \subsection{$H^2$-regularity}
In our error analysis we need the solution $u$ of \eqref{sto_evo} taking values in $H^2(D)^6$, which relies on additional regularity properties of ${\mathcal D}(M_0^2)={\mathcal D}(M^2)\cap {\mathbb V}_0$ and some smoothness of the coefficients $\varepsilon$ and $\mu$. Assume that 
\begin{equation}\label{assump_coe2}
\varepsilon,\mu \in W^{1,\infty}(D)\cap W^{2,3}(D),\quad\mbox{with }
\varepsilon, \mu \geq \delta \mbox{ for a constant } \delta>0.
\end{equation}
In fact, for any $w=({\bf E},{\bf H})\in {\mathcal D}(M_0^2)$, we already have $w\in H^1(D)^6$ from \eqref{embed H1}. Further, 
\[
M_0^2w=\begin{pmatrix}-\varepsilon^{-1}\nabla\times\big(\mu^{-1}\nabla\times{\bf E}\big)\\[0.5em] -\mu^{-1}\nabla\times\big(\varepsilon^{-1}\nabla\times{\bf H}\big)
\end{pmatrix}\in L^2(D)^6,
\]
 and the properties of curl operator lead to 
\begin{align*}
\Delta {\bf E}&=-\nabla\times(\nabla\times {\bf E})+\nabla(\nabla\cdot {\bf E})\\
&=-\mu\nabla\times(\mu^{-1}\nabla\times{\bf E})-\mu^{-1}\nabla\mu\times\big(\nabla\times{\bf E}\big)-\nabla(\varepsilon^{-1}\nabla\varepsilon\cdot {\bf E})\in L^2(D)^3,
\end{align*}
if the coefficients $\varepsilon,\; \mu$ satisfy \eqref{assump_coe2}. Then the $H^2$-regularity of ${\bf E}$ follows from  the equivalence of $H^2$-norm and the graph norm of Laplacian $\Delta$ on $D$ under certain mixed boundary conditions, i.e.,
if  there is a unique function $v\in H^1_{\Gamma}(D)$ solving
\[
\int_D v\phi {\rm d}{\bf x}+\int_D \nabla v\cdot\nabla\phi {\rm d}{\bf x}=\int_{D}f\phi {\rm d}{\bf x},
\]
for $f\in L^2(D)$ and  $\forall$ $\phi\in H^1_{\Gamma}(D)$, then the solution $v\in H^2(D)\cap H^1_{\Gamma}(D)$  satisfies $v-\Delta v=f$ on $D$, $\partial_{\bf n}v=0$ on $\partial D\backslash \Gamma$, and $\|v\|_{H^2(D)}\leq C\left(\|v\|_{L^2(D)}+\|\Delta v\|_{L^2(D)}\right)$ with the constant $C$ depending on $D$. Here for a union $\Gamma\subseteq \partial D$ of some faces of $D$, $H^1_{\Gamma}(D):=\{v\in H^1(D)|~~{\rm tr}(v)=0 \mbox{ on }\Gamma\}$.
For each component $E_j$ (resp. $H_j$) of ${\bf E}$ (resp. ${\bf H}$), the boundary $\Gamma$ may be taken as $\Gamma_{k}^{\pm}\cup \Gamma_{\ell}^{\pm}$ (resp. $\Gamma_j^{\pm}$) with $j,k,\ell\in \{1,2,3\}$ and $k\neq \ell\neq j$. We refer to \cite{HJS2015} for more details.

\begin{proposition}\label{exact_H2}
Let $Q^{\frac12}\in HS({\mathbb V},{\mathcal D}(M_0^2))$, and $u_0\in L^p(\Omega;{\mathcal D}(M_0^2))$ for some $p\geq 2$. Under the assumption \eqref{assump_coe2}, the solution \eqref{mild solution} has the following property
\begin{equation}\label{eq 2.7}
{\mathbb E}\bigg[\sup_{t\in[0,T]}\|u(t)\|_{H^2(D)^6}^{p}\bigg]\leq C{\mathbb E}\bigg[\sup_{t\in[0,T]}\|u(t)\|_{{\mathcal D}(M_0^2)}^{p}\bigg]
\leq C(1+{\mathbb E}\|u_0\|_{{\mathcal D}(M_0^2)}^{p}),
\end{equation}
where the constant $C$ depends on $T$, $\delta$, $\|\varepsilon\|_{W^{1,\infty}(D)}$, $\|\varepsilon\|_{W^{2,3}(D)}$, $\|\mu\|_{W^{1,\infty}(D)}$, $\|\mu\|_{W^{2,3}(D)}$  and $\|Q^{\frac12}\|_{HS({\mathbb V},\;{\mathcal D}(M_0^2))}$.
\end{proposition}
\begin{proof}
We first prove the ${\mathcal D}(M_0^2)$-regularity of the solution. From \eqref{mild solution}, we get
\begin{align}\label{eq 2.6}
&\Big[{\mathbb E}\big(\sup_{t\in[0,T]} \|u(t)\|_{{\mathcal D}(M_0^2)}^{p}\big)\Big]^{\frac{1}{p}}
\leq \Big[{\mathbb E}\big(\|u_0\|_{{\mathcal D}(M_0^2)}^{p}\big)\Big]^{\frac1p}\\
&+\bigg[{\mathbb E}\bigg(\sup_{t\in[0,T]}\Big\|\int_0^tS(t-s){\rm d}W(s)\Big\|_{{\mathcal D}(M_0^2)}^{p}\bigg) \bigg]^{\frac1p}
\leq  C\left(1+\Big[{\mathbb E}\big(\|u_0\|_{{\mathcal D}(M_0^2)}^{p}\big)\Big]^{\frac1p}\right),\nonumber
\end{align}
where the constant $C$ depends on $T$ and $\|Q^{\frac12}\|_{HS({\mathbb V},\;{\mathcal D}(M_0^2))}$.

The first inequality in \eqref{eq 2.7} comes from the embedding ${\mathcal D}(M_0^2)\hookrightarrow H^2(D)^6$. Thus the proof is finished by combining \eqref{eq 2.6}.
\end{proof}

\subsection{Small noise asymptotic behavior}

We scale the noise in the system \eqref{sto_evo}  by  a small  parameter $\sqrt{\lambda}$, $\lambda\in{\mathbb R}^{+}$,  i.e.,
\begin{equation}\label{sto_evo_para}
\begin{cases}
{\rm d}u(t)=Mu(t) {\rm d}t -\sqrt{\lambda}{\rm d}W(t),\\
u(0)=u_0,
\end{cases}
\end{equation}
whose mild solution is given by
$
u^{u_0,\lambda}(t)=S(t)u_0-\sqrt{\lambda}\int_0^tS(t-r){\rm d}W(r).
$
Denote the stochastic convolution $W_M(t)=\int_0^tS(t-r){\rm d}W(r)$. Then for arbitrary $T>0$, $W_M(T)$ is Gaussian on  ${\mathbb V}$ with mean $0$ and covariance operator
$
Q_{T}:={\rm Cov}\big(W_M(T)\big)=\int_0^T S(r)QS^{*}(r) {\rm d}r.
$

\begin{lemma}\cite[Proposition 12.10]{PZ2014}\label{lemma 2.1}
Assume that $X$ is a Gaussian random variable with distribution $\mu={\mathcal N}(0,\widetilde{Q})$ on a Hilbert space $H$. Then the family of random variables $\{X_{\lambda}:=\sqrt{\lambda}X\}_{\lambda>0}$ (or  measures $\left\{\mu_{\lambda}={\mathcal L}\big( X_{\lambda}\big)\right\}_{\lambda>0}$) satisfies the large deviation principle with the good rate function
\begin{equation}
I(x)=\begin{cases}
\frac12\|\widetilde{Q}^{-\frac12} x \|_{H}^2, & x\in \widetilde{Q}^{\frac12}({H}),\\[0.5em]
+\infty, & {\rm otherwise},
\end{cases}
\end{equation} 
where $\widetilde{Q}^{-\frac12}$ is the pseudo inverse of $\widetilde{Q}^{\frac12}$.
\end{lemma}

Based on Lemma \ref{lemma 2.1}, we get the following asymptotic behavior of the solution for \eqref{sto_evo_para} with small diffusion coefficient, which states that the laws of solutions satisfy the large deviation principle with the good rate function \eqref{rate function exact}.
\begin{proposition}\label{LDP_exact}
For arbitrary $T>0$ and $u_0\in{\mathbb V}$, the family of distributions $\left\{{\mathcal L}\big( u^{u_0,\lambda}(T) \big) \right\}_{\lambda>0}$ satisfies the  large deviation principle with the good rate function
\begin{equation}\label{rate function exact}
I_{T}^{u_0}(v)=\begin{cases}
\frac12\|Q_T^{-\frac12}\big( v-S(T)u_0 \big)\|_{\mathbb V}^2, & v-S(T)u_0\in Q_T^{\frac12}({\mathbb V}),\\[0.5em]
+\infty, & {\rm otherwise},
\end{cases}
\end{equation} 
where $Q_{T}^{-\frac12}$ is the pseudo inverse of $Q_T^{\frac12}$.
\end{proposition}
\begin{proof}
We define a process $Y^{\lambda}(t)=u^{u_0,\lambda}(t)-S(t)u_0$, which satisfies \eqref{sto_evo_para} with initial data $Y^{\lambda}(0)=0$. This means that $Y^{\lambda}(t)=\sqrt{\lambda}W_M(t)$. Then by the large deviation principle for Gaussian measures (Lemma \ref{lemma 2.1}), it follows that the good rate function of $\{Y^{\lambda}(T)\}_{\lambda>0}$ is given by
\begin{equation}\label{eq 2.13}
I_{T}^{0}(v)=\begin{cases}
\frac12\|Q_T^{-\frac12} v \|_{\mathbb V}^2, & v\in Q_T^{\frac12}({\mathbb V}),\\[0.5em]
+\infty, & {\rm otherwise}.
\end{cases}
\end{equation} 
In order to give the rate function of $\{u^{u_0,\lambda}(T)\}_{\lambda>0}$ based on \eqref{eq 2.13}, we use the definition of large deviation principle. Let $A\in{\mathcal B}({\mathbb V})$ be closed. Then $A-\{S(T)u_0\}$ still is closed in ${\mathcal B}({\mathbb V})$ and hence
\begin{align*}
&\limsup_{\lambda\to 0} \left[\lambda \ln {\mathbb P}\{ u^{u_0,\lambda}(T)\in A \}\right]
=\limsup_{\lambda\to 0} \left[\lambda \ln {\mathbb P}\{ Y^{\lambda}(T)\in A- \{S(T)u_0\}\}\right]\\
&\leq -\inf_{\overline{v}\in A- \{S(T)u_0\}} I_{T}^0(\overline{v})
=-\inf_{v\in A}I_{T}^{0}(v-S(T)u_0)=:-\inf_{v\in A}I_{T}^{u_0}(v).
\end{align*}
In a similar way we can check that for any open $B\in {\mathcal B}({\mathbb V})$, 
$$
\liminf_{\lambda\to 0} \left[\lambda \ln {\mathbb P}\{ u^{u_0,\lambda}(T)\in B \}\right]
\geq -\inf_{v\in B}I_{T}^{0}(v-S(T)u_0)=-\inf_{v\in B}I_{T}^{u_0}(v).
$$
Since $I_{T}^{u_0}$ fulfills the same properties as $I_{T}^{0}$, i.e. $I_{T}^{u_0}$ is a good rate function,  the proof is thus completed.
\end{proof}
\begin{remark}\label{remark 2.1}
If $Q$ commutes with $M$, then $Q_T^{\frac12}({\mathbb V})=Q^{\frac12}({\mathbb V})$.  In fact, 
$
Q_{T}=\int_0^T S(r)QS^{*}(r) {\rm d}r=TQ.
$
\end{remark}

\section{Temporal semidiscretization by stochastic symplectic method}\label{sec:3}
In this section, we study the semidiscretization in time of \eqref{sto_evo} by a midpoint scheme, which preserves the stochastic symplectic structure.  
The temporal semidiscretizations by a class of stochastic symplectic Runge-Kutta methods have been studied in \cite{CHJ2019b}. It is shown in there that the methods are convergent with order one in mean-square sense, if the solution has regularity in ${\mathcal D}(M^2)$.

For the time interval $[0,T]$, we introduce the uniform partition $0=t_0<t_1<\ldots<t_{N}=T$. Let $\tau=T/N$, and $\Delta W^{n+1}=W(t_{n+1})-W(t_n)$, $n=0,1,\ldots,N-1$. Applying the midpoint scheme to \eqref{sto_evo} in temporal direction yields
\begin{equation}\label{midpoint}
u^{n+1}=u^{n}+\frac{\tau}{2}(Mu^n+Mu^{n+1})-\Delta W^{n+1},
\end{equation}
which can also be written as
\begin{subequations}\label{midpoint2}
\begin{align}
\varepsilon{\bf E}^{n+1}=\varepsilon{\bf E}^{n}+\frac{\tau}{2}(\nabla\times{\bf H}^{n}+\nabla\times{\bf H}^{n+1})-\Delta W_{e}^{n+1},\\
\mu{\bf H}^{n+1}=\mu{\bf H}^{n}-\frac{\tau}{2}(\nabla\times{\bf E}^{n}+\nabla\times{\bf E}^{n+1})-\Delta W_{m}^{n+1}.
\end{align}
\end{subequations}
 
This scheme preserves the stochastic symplectic structure numerically, which is stated as follows.
 \begin{proposition}\cite[Theorem 4.3]{CHJ2019b}
 Let $\varepsilon,\mu$ be constants.
Under a zero boundary condition, the temporal semidiscretization \eqref{midpoint} preserves the discrete stochastic symplectic structure
		$\overline{\omega}^{n+1}=\int_{D}d{\bf E}^{n+1}\wedge d{\bf H}^{n+1}{\rm d}{\bf x} =\int_{D}{ d}{\bf E}^{n}\wedge {d}{\bf H}^{n}{\rm d}{\bf x}=\overline{\omega}^{n},~  {\mathbb P}\mbox{-a.s.}$
\end{proposition}		

The divergence conservation laws \eqref{sto_max3} can be preserved numerically by the temporal semidiscretization \eqref{midpoint}.	
\begin{proposition}
 For the temporal semidiscretization \eqref{midpoint}, if 
 $Qh\in{\mathbb V}_0$ for any $h\in{\mathbb V}$,
 then for any $n=0,1,\ldots, N-1$,
  \[
 \nabla\cdot(\varepsilon{\bf E}^{n+1})=\nabla\cdot(\varepsilon{\bf E}^{n}),
 \qquad \nabla\cdot(\mu{\bf H}^{n+1})=\nabla\cdot(\mu{\bf H}^{n}),\qquad   {\mathbb P}\mbox{-a.s.}
 \]
 \end{proposition}
 \begin{proof}
 The proof follows from the identity $\nabla\cdot (\nabla\times U)=0$ for  $U:~{\mathbb R}^3\to {\mathbb R}^3.$
 \end{proof}
 
 The solution of the temporal semidiscretization \eqref{midpoint} also has the same regularity as the exact solution of \eqref{sto_evo}, by using the embeddings ${\mathcal D}(M_0)\hookrightarrow H^1(D)^6$ and ${\mathcal D}(M_0^2)\hookrightarrow H^2(D)^6$. They are stated below without the proof.
\begin{proposition}\label{H1 of un}
 Under the conditions of Proposition \ref{exact_H1}, the solution of the  temporal semidiscretization \eqref{midpoint} has regularity in $H^1(D)^6$, and
\begin{equation}
\max_{0\leq n\leq N}{\mathbb E}\|u^n\|_{H^1(D)^6}^{p}
\leq C(1+{\mathbb E}\|u_0\|_{{\mathcal D}(M_0)}^{p}),
\end{equation}
where the constant $C$ depends on $T$, $\delta$, $\|\varepsilon\|_{W^{1,\infty}(D)}$, $\|\mu\|_{W^{1,\infty}(D)}$ and $\|Q^{\frac12}\|_{HS({\mathbb V},\;{\mathcal D}(M_0))}$.
\end{proposition}

\begin{proposition}\label{H2 of un}
 Under the conditions of Proposition \ref{exact_H2}, the solution of the  temporal semidiscretization \eqref{midpoint} has regularity in $H^2(D)^6$, and
\begin{equation}
\max_{0\leq n\leq N}{\mathbb E}\|u^n\|_{H^2(D)^6}^{p}
\leq C(1+{\mathbb E}\|u_0\|_{{\mathcal D}(M_0^2)}^{p}),
\end{equation}
where the constant $C$ depends on $T$, $\delta$, $\|\varepsilon\|_{W^{1,\infty}(D)}$, $\|\varepsilon\|_{W^{2,3}(D)}$, $\|\mu\|_{W^{1,\infty}(D)}$, $\|\mu\|_{W^{2,3}(D)}$  and $\|Q^{\frac12}\|_{HS({\mathbb V},\;{\mathcal D}(M_0^2))}$.
\end{proposition}

Let $S_{\tau}=\left(I-\frac{\tau}{2}M\right)^{-1}\left(I+\frac{\tau}{2}M\right)$ and $T_{\tau}=\left(I-\frac{\tau}{2}M\right)^{-1}$. The mild version of \eqref{midpoint} reads
\begin{equation}\label{discrete mild}
u^{n+1}=S_{\tau}u^n-T_{\tau}\Delta W^{n+1} =S_{\tau}^{n+1}u_0-\sum_{j=1}^{n+1}S_{\tau}^{n+1-j}T_{\tau}\Delta W^{j}.
\end{equation}
\begin{lemma}\label{lemma Ttau}
There exists a positive constant $C$ independent of $\tau$ such that
$\|I-T_{\tau}\|_{{\mathcal L}({\mathcal D}(M), {\mathbb V})}\leq C\tau.$
\end{lemma}
\begin{proof}
We define $\widetilde{v}=T_{\tau}v$ for any $v\in{\mathcal D}(M)$, which means that $\widetilde{v}=v+\frac{\tau}{2}M\widetilde{v}$. 
Taking inner product with $\widetilde{v}$ yields
$
\frac12\Big[\|\widetilde{v}\|_{\mathbb V}^2-\|v\|_{\mathbb V}^2+\|\widetilde{v}-v\|_{\mathbb V}^2 \Big]=\frac{\tau}{2}\langle M\widetilde{v}, \widetilde{v}\rangle_{\mathbb V} = 0.
$
Hence $\|\widetilde{v}\|_{\mathbb V}=\|T_{\tau}v\|_{\mathbb V}\leq \|v\|_{\mathbb V}$ leads to $\|T_{\tau}\|_{{\mathcal L}({\mathbb V},{\mathbb V})}\leq 1$.

The conclusion of this lemma is equivalent to $\|\widetilde{v}-v\|_{\mathbb V}\leq C\tau\|v\|_{{\mathcal D}(M)}$. In fact,
$
\|\widetilde{v}-v\|_{\mathbb V}=\frac{\tau}{2}\|M\widetilde{v}\|_{\mathbb V}=\frac{\tau}{2}\|T_{\tau}Mv\|_{\mathbb V}\leq \frac{\tau}{2}\|v\|_{{\mathcal D}(M)}.
$
Therefore the proof is finished.
\end{proof}

For the semigroups $S(t_n)$ and $S_{\tau}^{n}$, we have the following estimates.

\begin{lemma}\label{estimate semigroup}
For any integer $ n\in\{1,\ldots, N\}$, there exists a positive constant $C$ independent of $\tau$ such that
$ \|S(t_n)-S_{\tau}^n\|_{{\mathcal L}({\mathcal D}(M^k),{\mathbb V})}\leq C\tau^{k/2}~ \mbox{with }~k\in\{1,2\}.$
\end{lemma}
\begin{proof}
In order to estimate the error of semigroups, we denote $v(t)=S(t)v_0$ and 
$v^{k}=S_{\tau}^{k}v_0$. Then $\{v(t)\}_{t\in[0,T]}$ is the exact solution of $
\frac{\rm d}{{\rm d}t}v=M v,~ v(0)=v_0,
 $
while $\{v^k\}_{0\leq k\leq N}$ is the  solution of 
$
v^{k}=v^{k-1}+\frac{\tau}{2}\big(Mv^{k-1}+Mv^{k}\big),~ v^{0}=v_0.
$
Note that $v(t_k)=v(t_{k-1})+\int_{t_{k-1}}^{t_{k}}Mv(s){\rm d}s$ leads to
\[
e^{k}=e^{k-1}+\frac{\tau}{2}\big(Me^{k-1}+Me^{k}\big)+\int_{t_{k-1}}^{t_k} \Big[Mv(s)-\frac{1}{2}M v(t_{k-1})-\frac{1}{2}M v(t_{k})\Big]{\rm d}s,
\]
where $e^k=v(t_k)-v^k$.
Applying $\langle \cdot,~e^{k}+e^{k-1}\rangle_{\mathbb V}$ to both sides of the above equation, and using the skew-adjoint property of the operator $M$, we get
\begin{align}\label{eq 3.4}
\|e^k\|_{\mathbb V}^{2}&=\|e^{k-1}\|_{\mathbb V}^{2} +\int_{t_{k-1}}^{t_k} \Big\langle Mv(s)-\frac{1}{2}M v(t_{k-1})-\frac{1}{2}M v(t_{k}),~ e^{k}+e^{k-1}\Big\rangle_{\mathbb V} {\rm d}s\nonumber\\
&=\|e^{k-1}\|_{\mathbb V}^{2} -\frac12\int_{t_{k-1}}^{t_k} \Big\langle \int_{t_{k-1}}^{s} Mv(r){\rm d}r- \int_{s}^{t_{k}} Mv(r){\rm d}r,~ Me^{k}+Me^{k-1}\Big\rangle_{\mathbb V} {\rm d}s\\
&\leq \|e^{k-1}\|_{\mathbb V}^{2} +C\tau^2\Big(\sup_{t\in[0,T]}\|v(t)\|^2_{{\mathcal D}(M)}+ \max_{0\leq k\leq N}\|v^k\|^2_{{\mathcal D}(M)}\Big) \nonumber\\
&\leq  \|e^{k-1}\|_{\mathbb V}^{2}+C\tau^2\|v_0\|^2_{{\mathcal D}(M)},\nonumber
\end{align}
which yields $\max\limits_{1\leq k\leq N}\|e^k\|_{\mathbb V}=\max\limits_{1\leq k\leq N}\|\left(S(t_k)-S_{\tau}^{k}\right)v_0\|_{\mathbb V}\leq C\tau^{\frac12}\|v_0\|_{{\mathcal D}(M)}$.

In the other hand, based on \eqref{eq 3.4}, 
\begin{align*}
&\|e^k\|_{\mathbb V}^{2}=\|e^{k-1}\|_{\mathbb V}^{2} -\frac12\int_{t_{k-1}}^{t_k} \Big\langle \int_{t_{k-1}}^{s} Mv(r){\rm d}r- \int_{s}^{t_{k}} Mv(r){\rm d}r,~ Me^{k}+Me^{k-1}\Big\rangle_{\mathbb V} {\rm d}s\\
&=\|e^{k-1}\|_{\mathbb V}^{2} +\frac12\int_{t_{k-1}}^{t_k} \Big\langle \Big(\int_{t_{k-1}}^{s} \int_{t_{k-1}}^{r}- \int_{s}^{t_{k}}\int_{t_{k-1}}^{r} \Big)Mv(\xi){\rm d}\xi{\rm d}r,~ M^2(e^{k}+e^{k-1})\Big\rangle_{\mathbb V} {\rm d}s\\
&\leq  \|e^{k-1}\|_{\mathbb V}^{2}+C\tau^3\|v_0\|^2_{{\mathcal D}(M^2)},
\end{align*}
which yields $\max\limits_{1\leq k\leq N}\|e^k\|_{\mathbb V}=\max\limits_{1\leq k\leq N}\|\left(S(t_k)-S_{\tau}^{k}\right)v_0\|_{\mathbb V}\leq C\tau\|v_0\|_{{\mathcal D}(M^2)}$.
Therefore,  the proof is finished.
\end{proof}

\begin{theorem}\label{error of midpoint}
Let $Q^{\frac12}\in HS({\mathbb V},{\mathcal D}(M^k))$ and $u_0\in L^2(\Omega;{\mathcal D}(M^k))$ with $k\in\{1,2\}$. For the temporal semidiscretization \eqref{midpoint}, we have
	\begin{equation}
	\begin{split}
	\max_{1\leq n\leq N}\big(\mathbb{E}\|u(t_n)-u^n\|_{\mathbb V}^2\big)^{1/2}\leq C\tau^{k/2},\quad \mbox{for }k\in\{1,2\},
	\end{split}
	\end{equation}
	where the positive constant $C$ depends on $T$, $\|u_0\|_{L^2(\Omega;{\mathcal D}(M^k))}$ and $\|Q^{\frac12}\|_{HS({\mathbb V},{\mathcal D}(M^k))}$, but independent of $\tau$ and $n$.
\end{theorem}
\begin{proof}
From the mild solutions \eqref{mild solution}  and \eqref{discrete mild},
we use  It\^o isometry to get
\begin{align*}
&{\mathbb E}\|u(t_n)-u^n\|_{\mathbb V}^2\leq 2{\mathbb E}\|\big(S(t_n)-S_{\tau}^n\big)u_0\|^2_{\mathbb V}+2{\mathbb E}\bigg\| \sum_{j=1}^{n}\int_{t_{j-1}}^{t_j} \big(S(t_n-r)-S_{\tau}^{n-j}T_{\tau}\big) {\rm d}W \bigg\|^2_{\mathbb V}\\
&=2{\mathbb E}\|\big(S(t_n)-S_{\tau}^n\big)u_0\|^2_{\mathbb V}+2\sum_{j=1}^{n} \int_{t_{j-1}}^{t_j} \left\|\big(S(t_n-r)-S_{\tau}^{n-j}T_{\tau}\big)Q^{\frac12} \right\|^2_{HS({\mathbb V},{\mathbb V})}{\rm d}r. 
\end{align*}
The first term on the right-hand side is estimated by Lemma \ref{estimate semigroup}, and the second term on the right-hand side can be estimated by, for $r\in[t_{j-1},t_j]$,
\begin{align*}
\left\|\big(S(t_n-r)-S_{\tau}^{n-j}T_{\tau}\big)Q^{\frac12} \right\|_{HS({\mathbb V},{\mathbb V})}\leq \|S(t_n-t_j)\big(S(t_j-r)-I\big)Q^{\frac12}\|_{HS({\mathbb V},{\mathbb V})}\\
+\|\big(S(t_n-t_j)-S_{\tau}^{n-j}\big)Q^{\frac12}\|_{HS({\mathbb V},{\mathbb V})}+\|S_{\tau}^{n-j}\big(I-T_{\tau}\big)Q^{\frac12}\|_{HS({\mathbb V},{\mathbb V})}\\
\leq C\tau\|Q^{\frac12}\|_{HS({\mathbb V},{\mathcal D}(M))}+C\tau^{k/2}\|Q^{\frac12}\|_{HS({\mathbb V},{\mathcal D}(M^k))},\qquad \mbox{for }k\in\{1,2\},
\end{align*}
where in the last step, we use Lemmas \ref{lemma Ttau}-\ref{estimate semigroup} and \cite[Lemma 3.3]{CHJ2019a}. 
Combining them together, we finish the proof.
\end{proof}

Applying the midpoint scheme to discretize the system \eqref{sto_evo_para} with small noise, we get that
$
u^N=S_{\tau}^{N}u_0-\sqrt{\lambda}\sum_{j=1}^{N}S_{\tau}^{N-j}T_{\tau}\Delta W^{j}.
$
Let $W_{M;N}:=\sum_{j=1}^{N}S_{\tau}^{N-j}T_{\tau}\Delta W^{j}$. Then it is Gaussian on ${\mathbb V}$ with mean $0$ and covariance operator 
$
Q_{T;N}:={\rm Cov}(W_{M;N})=\tau\sum_{j=1}^{N}\big(S_{\tau}^{N-j}T_{\tau}\big)Q(S_{\tau}^{N-j}T_{\tau}\big)^{*}.
$
Analogously, as in Proposition \ref{LDP_exact}, we get the following result.
\begin{proposition}\label{LDP_semi}
For integer $N>0$ and $u_0\in{\mathbb V}$, the family of distributions $\{{\mathcal L}\big( u^{N;\;u_0,\lambda} \big)\}_{\lambda >0}$ satisfies the  large deviation principle with the good rate function
\begin{equation}
I_{T,N}^{u_0}(v)=\begin{cases}
\frac12\| \big(Q_{T;N}\big)^{-\frac12}\big(v-S_{\tau}^N u_0\big) \|_{\mathbb V}^2, & v-S_{\tau}^{N}u_0\in \big(Q_{T;N}\big)^{\frac12}({\mathbb V}),\\
+\infty,& {\rm otherwise}.
\end{cases}
\end{equation}
\end{proposition}
\begin{remark}\label{remark 3.1}
If $Q$ commutes with $M$, then $\big(Q_{T;N}\big)^{\frac12}({\mathbb V})=\big(T_{\tau}Q^{\frac12}\big)({\mathbb V})\subset Q^{\frac12}({\mathbb V})$. 
In fact,
$
Q_{T;N}=\tau\sum_{j=1}^{N}\big(S_{\tau}^{N-j}T_{\tau}\big)Q(S_{\tau}^{N-j}T_{\tau}\big)^{*}=\tau N T_{\tau}Q T_{\tau}^*=TT_{\tau}Q T_{\tau}^*
$
 yields the assertion.
\end{remark}

\begin{proposition}
Assume that  $Q$ commutes with $M$, and $v,\; u_0\in \big(T_{\tau}Q^{\frac12}\big)({\mathbb V})$, then there is a constant $C$ depending on $T$, $\|Q^{-\frac12}v\|_{{\mathcal D}(M)}$ and $\|Q^{-\frac12} u_0\|_{{\mathcal D}(M)}$ such that
$
\left| I_T^{u_0}(v)- I_{T,N}^{u_0}(v)\right|\leq C\tau^{\frac12}.
$

In addition, if $Q^{-\frac12}v$, $Q^{-\frac12}u_0\in {\mathcal D}(M^2)$, then there is a constant $C$ depending on $T$, $\|Q^{-\frac12}v\|_{{\mathcal D}(M^2)}$ and $\|Q^{-\frac12} u_0\|_{{\mathcal D}(M^2)}$ such that
$
\left| I_T^{u_0}(v)- I_{T,N}^{u_0}(v)\right|\leq C\tau.
$
\end{proposition}
\begin{proof}
Note that under the conditions of this proposition, 
\[
I_{T}^{u_0}(v)=\frac{1}{2T}\left\| Q^{-\frac12} \left(v-S(T)u_0\right)\right\|^2_{\mathbb V},\quad
I_{T,N}^{u_0}(v)=\frac{1}{2T}\left\| Q^{-\frac12}T_{\tau}^{-1} \left(v-S_{\tau}^N u_0\right)\right\|^2_{\mathbb V}.
\]
Thus,
\begin{align}\label{eq 3.6}
&\left| I_{T}^{u_0}(v)-I_{T,N}^{u_0}(v)\right|\\
&=\frac{1}{2T}\Big|\Big\langle
Q^{-\frac12} \left(v-S(T)u_0\right)+Q^{-\frac12}T_{\tau}^{-1} \left(v-S_{\tau}^N u_0\right),\nonumber\\
&\qquad\qquad Q^{-\frac12} \left(v-S(T)u_0\right) -Q^{-\frac12}T_{\tau}^{-1} \left(v-S_{\tau}^N u_0\right)\Big\rangle_{\mathbb V}\Big| \nonumber\\
&\leq  C
\left\|Q^{-\frac12} \left(v-S(T)u_0\right) -Q^{-\frac12}T_{\tau}^{-1} \left(v-S_{\tau}^N u_0\right)\right\|_{\mathbb V}\nonumber\\
&\leq C \left[\left\|Q^{-\frac12} (I-T_{\tau}^{-1})\left(v-S(T)u_0\right) \right\|_{\mathbb V}+\left\|Q^{-\frac12} T_{\tau}^{-1}\left(S(T)-S_{\tau}^N\right)u_0 \right\|_{\mathbb V}\right],\nonumber
\end{align}
where the constant $C$ depends on $T,\; \| Q^{-\frac12}T_{\tau}^{-1}v\|_{\mathbb V},\; \| Q^{-\frac12}T_{\tau}^{-1}u_0\|_{\mathbb V}$.
Since $I-T_{\tau}^{-1}=\frac{\tau}{2}M$, 
\begin{align*}
\left\|Q^{-\frac12} (I-T_{\tau}^{-1})\left(v-S(T)u_0\right) \right\|_{\mathbb V}
\leq C(\|Q^{-\frac12}Mv\|_{\mathbb V},\; \|Q^{-\frac12}Mu_0\|_{\mathbb V})\tau.
\end{align*}
And for the second term on the right-hand side of \eqref{eq 3.6},
\begin{align*}
&\left\|Q^{-\frac12} T_{\tau}^{-1}\left(S(T)-S_{\tau}^N\right)u_0 \right\|_{\mathbb V}\\
&\leq \left\|Q^{-\frac12} \left(S(T)-S_{\tau}^N\right)u_0 \right\|_{\mathbb V}+\frac{\tau}{2}\left\|Q^{-\frac12} M\left(S(T)-S_{\tau}^N\right)u_0 \right\|_{\mathbb V}\\
&\leq \left\|Q^{-\frac12} \left(S(T)-S_{\tau}^N\right)u_0 \right\|_{\mathbb V}+C(\|Q^{-\frac12}Mu_0\|_{\mathbb V})\tau.
\end{align*}
Lemma \ref{estimate semigroup} yields the conclusion and thus the proof is finished.
\end{proof}
\section{Spatial semidiscretization by dG method}\label{sec4}

In this section, we investigate the semidiscretization of the stochastic Maxwell equations \eqref{sto_evo} in space by  the dG method with the upwind fluxes, including the properties of the discrete Maxwell operator, the well-posedness of the spatial semidiscretization, the preservation of the divergence properties in a weak sense, and the mean-square error estimate of the  semidiscrete method in space.


\subsection{Discrete Maxwell operator} 
The notations and properties of the discrete Maxwell operator are based on \cite{HP2015}. 
Let ${\mathcal T}_h=\{K\}$ be a simplicial, shape- and contact-regular mesh
of the domain $D$ consisting of elements $K$, i.e., $D=\bigcup K$. 
The index $h$ refers to the maximum diameter of all elements of ${\mathcal T}_h$. 
The dG space with respect to the mesh ${\mathcal T}_h$
is taken to be the set of piecewise linear functions,
i.e., 
$
{\mathbb V}_h:={\mathbb P}_1({\mathcal T}_h)^6:=\{v_h\in L^2(D):~v_h|_{K}\in {\mathbb P}_1(K)\}^6,
$
where ${\mathbb P}_1(K)$ denotes the set of continuous piecewise polynomials of degree $\leq 1$. In general, ${\mathbb V}_h \not\subset {\mathcal D}(M_0)$.
The set of faces is denoted by ${\mathcal G}_h={\mathcal G}_h^{\rm int}\cup {\mathcal G}_h^{\rm ext}$, where ${\mathcal G}_h^{\rm int}$ and ${\mathcal G}_h^{\rm ext}$ consist of all interior and all exterior faces, respectively. By ${\bf n}_{F}$ we denote the unit normal of a face $F\in{\mathcal G}_h^{\rm int}$, where the orientation of ${\bf n}_{F}$ is fixed once and forever for each inner face. And for a boundary face $F\in {\mathcal G}_h^{\rm ext}$, ${\bf n}_{F}$ is an outward normal vector. The broken Sobolev spaces are defined by
$
H^k({\mathcal T}_h):=\{ v\in L^2(D):~v|_{K}\in H^k(K) \mbox{ for all } K\in {\mathcal T}_h \},~~k\in {\mathbb N},
$
with seminorm and norm being $|v|_{H^k({\mathcal T}_h)}^2:=\sum\limits_{K\in {\mathcal T}_h}|v|^2_{H^k(K)}$ and $\|v\|^2_{H^k({\mathcal T}_h)}:=\sum\limits_{j=0}^{k}|v|^2_{H^j({\mathcal T}_h)}$, respectively. Note that $H^k(D)\subset H^k({\mathcal T}_h)$.

\begin{assumption}\label{assump}
Assume that $\pi_h:~{\mathbb V}\to {\mathbb V}_h$ is the orthogonal  projection, defined by, for every $v\in{\mathbb V}$,
 \begin{equation}\label{property of projection}
 \langle v-\pi_h v,~u_h\rangle_{\mathbb V}=0\qquad \mbox{for all}~~  u_h\in {\mathbb V}_h.
 \end{equation}
 Moreover, for all $v\in H^s({\mathcal T}_h)^6$ with integer $s\leq 2$, it holds that
\begin{equation}\label{projection order}
\|v-\pi_h v\|_{\mathbb V}\leq Ch^{s} |v|_{H^s({\mathcal T}_h)^6},
\end{equation}
and 
\begin{equation}\label{projection order boundary}
\sum_{F\in {\mathcal G}_h }\|v-\pi_{h}v\|^2_{L^2(F)^6}\leq Ch^{2s-1}|v|^2_{H^s({\mathcal T}_h)^6},
\end{equation}
where the constant $C$ is independent of  $h$.
\end{assumption}

\begin{remark}
\begin{itemize}
\item[(i)] For the projection operator $\pi_h$ in Assumption \ref{assump}, it is not difficult to get that  $\|\pi_h v\|_{\mathbb V}\leq \|v\|_{\mathbb V}$.
\item[(ii)] Suppose that $\mu_{K}:=\mu|_{K}$ and $\varepsilon_K:=\varepsilon|_{K}$ are constants for each $K\in {\mathcal T}_h$, then the usual $L^2$-orthogonal projection $\pi_h$ on ${\mathbb P}_1({\mathcal T}_h)$ satisfies Assumption \ref{assump}, where the projection acts componentwise for vector fields.
\end{itemize}
\end{remark}

Define by $[[v]]_{F}:=\big(v_{K_{F}}\big)|_{F}-\big(v_{K}\big)|_{F}$ the jump of $v$ on an interior face $F$ with normal vector ${\bf n}_{F}$ pointing from $K$ to $K_{F}$. The Maxwell operator discretized by a dG method with the upwind fluxes is defined as follows.

\begin{definition}\label{def discrete Mh}
Given $u_h=({\bf E}_h^{\top},{\bf H}_h^{\top})^{\top}$,  $v_h=(\psi_h^{\top},\phi_h^{\top})^{\top}\in{\mathbb V}_h$, the discrete Maxwell operator $M_h:{\mathbb V}_h\to {\mathbb V}_h$ is given as
\begin{align*}\label{discrete Mh}
&\langle M_h u_h, v_h \rangle_{\mathbb V}:=\sum_{K}\Big( \langle \nabla\times {\bf H}_h, \psi_h\rangle_{L^2(K)^3}-\langle \nabla\times {\bf E}_h, \phi_h\rangle_{L^2(K)^3} \Big)\nonumber\\
&+\sum_{F\in {\mathcal G}_h^{\rm int}}\Big(
\langle {\bf n}_F\times [[{\bf H}_h]]_{F},\beta_K \psi_K+\beta_{K_F}\psi_{K_F} \rangle_{L^2( F)^3}\nonumber\\
&- \langle {\bf n}_F\times [[{\bf E}_h]]_{F},\alpha_K \phi_K+\alpha_{K_F}\phi_{K_F} \rangle_{L^2( F)^3} \nonumber\\
&-\gamma_{F} \langle {\bf n}_F\times [[{\bf E}_h]]_{F}, {\bf n}_F\times [[\psi_h]]_{F} \rangle_{L^2( F)^3}
-\delta_{F}\langle {\bf n}_F\times [[{\bf H}_h]]_{F}, {\bf n}_F\times [[\phi_h]]_{F} \rangle_{L^2( F)^3}\Big)\nonumber\\
&+\sum_{F\in{\mathcal G}_h^{\rm ext}} \Big(
\langle {\bf n}\times {\bf E}_h, \phi_h\rangle_{L^2( F)^3}
-2\gamma_{F} \langle {\bf n}\times {\bf E}_h, n\times\psi_h\rangle_{L^2( F)^3}
 \Big),
\end{align*}
where 
\begin{align*}
\alpha_{K}=\frac{C_{K_{F}}\varepsilon_{K_F}}{C_{K_{F}}\varepsilon_{K_F}+C_K \varepsilon_{K}},\qquad
\beta_{K}=\frac{C_{K_{F}}\mu_{K_F}}{C_{K_{F}}\mu_{K_F}+C_K \mu_{K}},\\
\gamma_F=\frac{1}{C_{K_{F}}\mu_{K_F}+C_K \mu_{K}},\qquad
\delta_{F}=\frac{1}{C_{K_{F}}\varepsilon_{K_F}+C_K \varepsilon_{K}},
\end{align*}
with $C_{K}=(\varepsilon_{K}\mu_{K})^{-1/2}$.
\end{definition}

The discrete Maxwell operator $M_h$ is also well-defined as an operator from ${\mathbb V}_h+\big({\mathcal D}(M)\cap H^1({\mathcal T}_h)^6\big)$ to ${\mathbb V}_h$, and has the following properties. Here ${\mathbb V}_h+\big({\mathcal D}(M)\cap H^1({\mathcal T}_h)^6\big):=\{v_h+u:\; v_h\in{\mathbb V}_h, \; u\in {\mathcal D}(M)\cap H^1({\mathcal T}_h)^6\}$. We refer to \cite[Lemmas 4.3-4.5]{HP2015} for proofs.
\begin{proposition}\label{properties of discrete dG}
 \begin{itemize}
\item[(i)] For $u\in {\mathcal D}(M)\cap H^1({\mathcal T}_h)^6$, we have $M_hu=\pi_h M u$.
\item[(ii)] For all $u_h=({\bf E}_h^{\top},{\bf H}_h^{\top})^{\top}\in{\mathbb V}_h$, we have
\begin{align*}
\langle M_h u_h, u_h\rangle_{\mathbb V}=&-\sum_{F\in{\mathcal G}_h^{int}}\Big(
\gamma_{F}\|{\bf n}_{F}\times[[{\bf E}_h]]_F\|^2_{L^2(F)^3}+\delta_{F}\|{\bf n}_{F}\times[[{\bf H}_h]]_F\|^2_{L^2(F)^3}
\Big)\\
&-2\sum_{F\in{\mathcal G}_h^{ext}}\gamma_{F} \|{\bf n}_{F}\times {\bf E}_h\|^2_{L^2(F)^3}\geq 0.
\end{align*}
In particular, $M_h$ is dissipative on ${\mathbb V}_h$.
\item[(iii)] For $u=({\bf E}^{\top},{\bf H}^{\top})^{\top}\in {\mathbb V}_h+\big({\mathcal D}(M)\cap H^1({\mathcal T}_h)^6\big)$ and $v_h=(\psi_h^{\top},\phi_h^{\top})^{\top}\in{\mathbb V}_h$, we have
\begin{align*}
&\langle M_h u, v_h\rangle_{\mathbb V}=\sum_{K}\Big( \langle {\bf H}, \nabla\times \psi_h\rangle_{L^2(K)^3}-\langle {\bf E}, \nabla\times \phi_h\rangle_{L^2(K)^3} \Big)\\
&+\sum_{F\in{\mathcal G}_h^{int}}\Big(
\langle \beta_K{\bf H}_{K_F}+\beta_{K_F}{\bf H}_K-\gamma_F {\bf n}_F\times [[{\bf E}]]_{F}, {\bf n}_F\times [[\psi_h]]_{F} \rangle_{L^2(F)^3}\\
&-\langle \alpha_K{\bf E}_{K_F}+\alpha_{K_F}{\bf E}_K+\delta_F {\bf n}_F\times [[{\bf H}]]_{F}, {\bf n}_F\times [[\phi_h]]_{F} \rangle_{L^2(F)^3}
\Big)\\
&+\sum_{F\in{\mathcal G}_h^{ext}} \langle {\bf H}, {\bf n}_F\times\psi_h\rangle_{L^2(F)^3}-2\gamma_{F}
\langle {\bf n}_{F}\times {\bf E}, {\bf n}_F\times\psi_h\rangle_{L^2(F)^3}.
\end{align*}
\end{itemize}
\end{proposition}

\subsection{Semidiscrete method in space}
After discretizing \eqref{sto_evo} by a dG method with the upwind fluxes, we end up with the spatial semidiscretization
\begin{equation}\label{dG}
\begin{cases}
{\rm d}u_h(t)=M_h u_h(t){\rm d}t-\pi_h{\rm d}W(t),\\
u_h(0)=\pi_h u_0,
\end{cases}
\end{equation}
where $M_h$ is the discrete Maxwell operator in Definition \ref{def discrete Mh}, and $u_h(t)\in {\mathbb V}_h$ is an approximation of the exact solution $u(t)\in{\mathbb V}$.

Notice that the equation \eqref{dG} actually  is a finite dimensional stochastic differential equation. In fact, let $\{\phi_1,\ldots, \phi_{N_h}\}$ be a basis for ${\mathbb V}_h$. Utilizing this basis, the  semidiscrete problem \eqref{dG} in space can be rewritten as, for $j=1,\ldots, N_h$,
\begin{equation}
\begin{cases}
{\rm d}\langle u_h(t),\phi_j \rangle_{\mathbb V}=\langle M_hu_h(t),\phi_j \rangle_{\mathbb V}{\rm d}t-\langle  \phi_j,{\rm d}W(t) \rangle_{\mathbb V},\\
\langle u_h(0), \phi_j\rangle_{\mathbb V}=\langle u_0, \phi_j\rangle_{\mathbb V}.
\end{cases}
\end{equation}
Since $u_h(t)\in L^2(\Omega;{\mathbb V}_h)$, we get
$
u_h(t)=\sum\limits_{\ell=1}^{N_h}u_{[\ell]}(t)\phi_{\ell}.
$
Denoting  $A=\Big(\langle \phi_{\ell},\phi_{j}\rangle_{\mathbb V}\Big)_{j,\ell}\in{\mathbb R}^{N_h \times N_h}$, $B=\Big(\langle M_h\phi_{\ell},\phi_{j}\rangle_{\mathbb V}\Big)_{j,\ell}\in{\mathbb R}^{N_h \times N_h}$, ${\bf u}(t)=(u_{[1]}(t),\ldots, u_{[N_h]}(t))^{\top}\in{\mathbb R}^{N_h}$, ${\bf u}_0=(\langle u_0, \phi_1\rangle_{\mathbb V},\ldots, \langle u_0, \phi_{N_h}\rangle_{\mathbb V})^{\top}\in{\mathbb R}^{N_h}$ and ${\bf W}(t)=(W_{[1]}(t),\ldots, W_{[N_h]}(t))^{\top}\in{\mathbb R}^{N_h}$ with $W_{[j]}(t)=\langle  \phi_j,W(t) \rangle_{\mathbb V}$, we obtain the system of stochastic ordinary differential equations on ${\mathbb R}^{N_h}$ for \eqref{dG},
\begin{equation}\label{SODE}
\begin{cases}
A{\rm d}{\bf u}(t)=B{\bf u}(t){\rm d}t-{\rm d}{\bf W}(t),\\
A{\bf u}(0)={\bf u}_0.
\end{cases}
\end{equation}
Notice that the components of ${\bf W}(t)$ are correlated with
\[
{\mathbb E}\big(W_{[j]}(t)W_{[\ell]}(t)\big)={\mathbb E}\left(\langle  \phi_j,W(t) \rangle_{\mathbb V}\langle  \phi_{\ell},W(t) \rangle_{\mathbb V}\right)
=t\langle  Q\phi_j,\phi_{\ell} \rangle_{\mathbb V},~~\forall~j,\ell=1,\ldots,N_h.
\]
\begin{proposition}
The spatially semidiscrete problem \eqref{dG} is well-posed, i.e., there is a unique solution $u_h\in L^2(\Omega; C([0,T];{\mathbb V}_h))$ given by
\begin{equation}\label{dG mild solution}
u_h(t)=e^{tM_h}u_h(0)-\int_0^te^{(t-s)M_h}\pi_h {\rm d}W(s).
\end{equation}
Moreover, we have
\begin{equation}\label{eq 4.9 bound}
{\mathbb E}\Big[\sup_{t\in[0,T]}\|u_h(t)\|^2_{{\mathbb V}}\Big]\leq C(1+{\mathbb E}\|u_0\|_{\mathbb V}^2),
\end{equation}
where the constant $C$ depends on $T$ and ${\rm Tr}(Q)$.
\end{proposition}
\begin{proof}
Note that $I-M_h:~{\mathbb V}_h\to {\mathbb V}_h$ is injective and surjective, and thus ${\rm Ran}(I-M_h)={\mathbb V}_h$. Since  the discrete operator $M_h$ is dissipative on ${\mathbb V}_h$, it generates a contraction semigroup. Therefore, the unique solution of \eqref{dG} is given by \eqref{dG mild solution}.

The estimate in \eqref{eq 4.9 bound} is obtained by the triangle inequality and the estimate on stochastic convolution
\begin{align*}
&{\mathbb E}\Big[\sup_{t\in[0,T]}\|u_h(t)\|^2_{{\mathbb V}}\Big]\\
&\leq 2{\mathbb E}\Big[\sup_{t\in[0,T]}\|e^{tM_h}u_h(0)\|^2_{{\mathbb V}}\Big]+2{\mathbb E}\Big[\sup_{t\in[0,T]}\Big\|\int_0^te^{(t-s)M_h}\pi_h {\rm d}W(s)\Big\|^2_{{\mathbb V}}\Big]\\
&\leq 2{\mathbb E}\|u_h(0)\|^2_{{\mathbb V}}+2T{\mathbb E}\|\pi_{h}Q^{\frac12}\|_{HS({\mathbb V},{\mathbb V})}^2\leq C(1+{\mathbb E}\|u_0\|_{\mathbb V}^2),
\end{align*}
where in the last step we use the property $\|\pi_h u\|_{\mathbb V}\leq \|u\|_{\mathbb V}$ of the projection operator. Thus the proof is finished.
\end{proof}

It is not difficult to observe that $W_{M;h}(t)=\int_0^t e^{(t-s)M_h}\pi_h{\rm d}W(s)$ is Gaussian on ${\mathbb V}_h$ with mean $0$ and covariance operator
$$
Q_{T,h}:={\rm Cov}\big( W_{M;h}(T)\big)=\int_0^T \big(e^{rM_h}\pi_h\big)Q\big(e^{rM_h}\pi_h\big)^{*} {\rm d}r.
$$

Applying the dG method to discretize the spatial direction of the small noise system \eqref{sto_evo_para}, we denote by $\{{\mathcal L}\big( u_h^{u_0,\lambda}(T) \big)\}_{\lambda>0}$ the laws of the semidiscrete solutions.
The asymptotic behavior of $\{{\mathcal L}\big( u_h^{u_0,\lambda}(T) \big)\}_{\lambda>0}$ is similar to that of $\{{\mathcal L}\big( u^{u_0,\lambda}(T) \big)\}_{\lambda>0}$ in Proposition \ref{LDP_exact}, which is stated below.
\begin{proposition}
For arbitrary $T>0$ and $u_0\in{\mathbb V}$, the family of distributions $\left\{{\mathcal L}\big( u_h^{u_0,\lambda}(T) \big) \right\}_{\lambda>0}$ satisfies the  large deviation principle with the good rate function
\begin{equation}
I_{T,h}^{u_0}(v)=\begin{cases}
\frac12\|Q_{T,h}^{-\frac12}\big( v-e^{TM_h}\pi_hu_0 \big)\|_{\mathbb V}^2, &v-e^{TM_h}\pi_hu_0\in Q_{T,h}^{\frac12}({\mathbb V}_h),\\[0.5em]
+\infty, & {\rm otherwise},
\end{cases}
\end{equation} 
where $Q_{T,h}^{-\frac12}$ is the pseudo inverse of $Q_{T,h}^{\frac12}$.
\end{proposition}

\subsection{Discrete divergence conservation property}
If $u_0\in{\mathbb V}_0$ and $Q^{\frac12}\in HS({\mathbb V},{\mathbb V}_0)$, the exact solution $u(t)$ of the stochastic Maxwell equations \eqref{sto_evo} possesses the divergence relations \eqref{sto_max3}: $\nabla\cdot(\varepsilon {\bf E}(t))=0$ and $\nabla\cdot(\varepsilon {\bf H}(t))=0$. However, for the spatial semidiscretization \eqref{dG}, we prove that  the divergence relations is preserved numerically in the following discrete weak sense.

Define the test space $X_h\subset H_0^1(D)$ as
 $X_h:=\{v\in C^0(\bar{D}): v_h|_{K}\in {\mathbb P}_2(K),~K\in{\mathcal T}_h\}\cap H_0^1(D).$ By $\langle\cdot, \cdot\rangle_{-1}$ we denote the duality product between $H^{-1}(D)$ and $H_0^1(D)$, in which
$ 
 \langle\nabla\cdot E, \psi\rangle_{-1}= -\langle E, \nabla\psi\rangle_{L^2(D)^3},~ \forall~E\in L^2(D)^3,~\psi\in H_0^1(D).
$
 \begin{proposition}\label{prop: divergence free}
 Let $u_0\in{\mathbb V}_0$ and $Q^{\frac12}\in HS({\mathbb V},{\mathbb V}_0)$. The solution $({\bf E}_h(t),{\bf H}_h(t))$ of the spatially semidiscrete problem \eqref{dG} satisfies:  $\forall ~t\in[0,T]$, and $\forall ~\phi\in X_h$, 
 \[
 \langle \nabla\cdot(\varepsilon {\bf E}_h(t)), \phi\rangle_{-1}=
  \langle \nabla\cdot(\mu {\bf H}_h(t)), \phi\rangle_{-1}=0,\qquad \qquad {\mathbb P}\mbox{-a.s.}
 \]
 \end{proposition}
 \begin{proof}
 For $\psi,\phi\in X_h$, using the definition of the duality product $\langle\cdot, \cdot\rangle_{-1}$, we get
 \begin{align*}
\left\langle\begin{pmatrix}\nabla\cdot(\varepsilon{\bf E}_h(t))\\ \nabla\cdot(\mu{\bf H}_h(t))\end{pmatrix}, \begin{pmatrix} \psi \\ \phi\end{pmatrix}\right\rangle_{-1}
&=\langle \nabla\cdot(\varepsilon{\bf E}_h(t)),  \psi\rangle_{-1}
+\langle \nabla\cdot(\varepsilon{\bf H}_h(t)),  \phi\rangle_{-1}\\
&=-\langle \varepsilon{\bf E}_h(t), \nabla\psi\rangle_{L^2(D)^3}
-\langle \varepsilon{\bf H}_h(t), \nabla\phi\rangle_{L^2(D)^3}\\
&=-\left\langle\begin{pmatrix}{\bf E}_h(t)\\ {\bf H}_h(t)\end{pmatrix}, \begin{pmatrix} \nabla\psi \\ \nabla\phi\end{pmatrix}\right\rangle_{\mathbb V}.
 \end{align*}
 Using \eqref{dG} we obtain
 \begin{align*}
 \left\langle\begin{pmatrix}{\bf E}_h(t)\\ {\bf H}_h(t)\end{pmatrix}, \begin{pmatrix} \nabla\psi \\ \nabla\phi\end{pmatrix}\right\rangle_{\mathbb V}
 =&\left\langle\begin{pmatrix}{\bf E}_h(0)\\ {\bf H}_h(0)\end{pmatrix}, \begin{pmatrix} \nabla\psi \\ \nabla\phi\end{pmatrix}\right\rangle_{\mathbb V}
 +\int_0^t \left\langle M_h\begin{pmatrix}{\bf E}_h(s)\\ {\bf H}_h(s)\end{pmatrix}, \begin{pmatrix} \nabla\psi \\ \nabla\phi\end{pmatrix}\right\rangle_{\mathbb V}{\rm d}s\\
& -\left\langle\pi_h\begin{pmatrix} \varepsilon^{-1}W_e(t)\\  \mu^{-1}W_m(t)\end{pmatrix}, \begin{pmatrix} \nabla\psi \\ \nabla\phi\end{pmatrix}\right\rangle_{\mathbb V}.
 \end{align*}
 For the first and third terms on the right-hand side, we utilize the property  \eqref{property of projection} of projection and the fact that  $ \begin{pmatrix} \nabla\psi \\ \nabla\phi\end{pmatrix}\in {\mathbb V}_h$ to get
 \[
 \left\langle\begin{pmatrix}{\bf E}_h(0)\\ {\bf H}_h(0)\end{pmatrix}, \begin{pmatrix} \nabla\psi \\ \nabla\phi\end{pmatrix}\right\rangle_{\mathbb V}
 =\left\langle \pi_h\begin{pmatrix}{\bf E}(0)\\ {\bf H}(0)\end{pmatrix}, \begin{pmatrix} \nabla\psi \\ \nabla\phi\end{pmatrix}\right\rangle_{\mathbb V}
 =\left\langle \begin{pmatrix}{\bf E}(0)\\ {\bf H}(0)\end{pmatrix}, \begin{pmatrix} \nabla\psi \\ \nabla\phi\end{pmatrix}\right\rangle_{\mathbb V}=0,
 \]
 and
 \begin{align*}
 \left\langle\pi_h\begin{pmatrix} \varepsilon^{-1}W_e(t)\\  \mu^{-1}W_m(t)\end{pmatrix}, \begin{pmatrix} \nabla\psi \\ \nabla\phi\end{pmatrix}\right\rangle_{\mathbb V}
& = \left\langle\begin{pmatrix} \varepsilon^{-1}W_e(t)\\ \mu^{-1}W_m(t)\end{pmatrix}, \begin{pmatrix} \nabla\psi \\ \nabla\phi\end{pmatrix}\right\rangle_{\mathbb V}\\
& =\langle W_e(t), \nabla\psi\rangle_{L^2(D)^3}+\langle W_m(t), \nabla\phi\rangle_{L^2(D)^3}\\
& =-\langle \nabla\cdot W_e(t), \psi\rangle_{-1}-\langle \nabla\cdot W_m(t),\phi\rangle_{-1}=0.
  \end{align*}
 Using Proposition \ref{properties of discrete dG} (iii),  the second term on the right-hand side equals to zero, since for any function $\varphi \in X_h$, we have $\nabla\times\nabla \varphi={\bf 0}$, ${\bf n}_{F}\times [[\nabla \varphi]]_{F}={\bf 0}$ for $F\in {\mathcal G}_h^{\rm int}$ and ${\bf n}\times \nabla\varphi={\bf 0}$ on $\partial D$. 
   Therefore, the conclusion of this proposition comes from taking  $\phi=0$ or $\psi=0$, respectively.
 \end{proof}
 
 \begin{remark}\label{remark1}
 The projection of the exact solution of \eqref{sto_evo} has the same property, $\forall ~t\in[0,T]$, and $\forall ~\phi\in X_h$, 
$$
 \langle \nabla\cdot\pi_h(\varepsilon {\bf E}(t)), \phi\rangle_{-1}=
  \langle \nabla\cdot\pi_h(\mu {\bf H}(t)), \phi\rangle_{-1}=0,\qquad ~ {\mathbb P}\mbox{-a.s.}
$$
 In fact, since $\nabla\phi\in {\mathbb V}_h$, we have
$
 \langle \nabla\cdot\pi_h(\varepsilon {\bf E}(t)), \phi\rangle_{-1}
 =\langle \pi_h(\varepsilon {\bf E}(t)), \nabla\phi\rangle_{L^2(D)^3}
 =\langle \varepsilon {\bf E}(t), \nabla\phi\rangle_{L^2(D)^3}
 =\langle \nabla\cdot(\varepsilon {\bf E}(t)), \phi\rangle_{-1}=0.
$
 
 \end{remark}
 
\subsection{Error estimate of spatial semidiscretization}
To investigate the error of the spatial semidiscretization \eqref{dG}, we apply the projection $\pi_h$ to the continuous problem \eqref{sto_evo} and use  Proposition \ref{properties of discrete dG} (i) to get
\begin{equation}\label{projection sto_evo}
{\rm d}\pi_h u(t)=M_h u(t){\rm d}t-\pi_h {\rm d}W(t),\qquad \pi_h u(0)=\pi_h u_0.
\end{equation}
We define the error $e(t)=u_h(t)-u(t)=\big(u_h(t)-\pi_h u(t)\big)-\big(u(t)-\pi_h u(t)\big)=:e_h(t)-e_{\pi}(t)$.

The mean-square error estimate of the spatial semidiscretization \eqref{dG} is given in the following theorem. 
\begin{theorem}\label{thm: error of eh}
Let  $u\in C([0,T];L^2(\Omega;  H^{k}(D)^6))$ with $k\in\{1,2\}$ be the solution of \eqref{sto_evo} and let $u_h\in C([0,T];L^2(\Omega; {\mathbb V}_h))$ be the solution of \eqref{dG}. Then there is a constant $C$ independent of $h$ such that
$
\sup_{t\in[0,T]}\Big({\mathbb E}\|u_h(t)-u(t)\|^2_{\mathbb V}\Big)^{\frac12}\leq Ch^{k-\frac12} ~\mbox{for }~ k\in\{1,2\}.
$
\end{theorem}
\begin{proof}
For the part $e_{\pi}(t)$, by using \eqref{projection order}, we have
\begin{align}\label{eq 4.14}
{\mathbb E}\|e_{\pi}(t)\|^2_{\mathbb V}={\mathbb E}\|u(t)-\pi_h u(t)\|^2_{\mathbb V}
\leq Ch^{2k}{\mathbb E}|u(t)|^2_{H^{k}(D)^6}.
\end{align}
For the part $e_{h}(t)$, we subtract \eqref{projection sto_evo} from \eqref{dG} to get
$
{\rm d}e_h(t)=M_h e_h(t){\rm d}t-M_h e_{\pi}(t){\rm d}t,~e_h(0)=0.
$
Then we obtain, for any $t\in[0,T]$,
$$
\frac12 \|e_h(t)\|^2_{\mathbb V} -\int_0^t \langle M_he_h(s), e_h(s)
\rangle_{\mathbb V}{\rm d}s =-\int_0^t\langle M_he_{\pi}(s), e_h(s)
\rangle_{\mathbb V}{\rm d}s.
$$

For the term on the right-hand side, noticing $e_{h}(s)\in{\mathbb V}_h$, $e_{\pi}(s)\in {\mathbb V}_h+({\mathcal D}(M)\cap H^1(D)^6 )$, we use  Proposition \ref{properties of discrete dG} (iii) to obtain
\begin{align*}
&|\langle M_h e_{\pi}, e_h\rangle_{\mathbb V}|=\sum_{K}\Big( \langle e_{\pi}^{\bf H}, \nabla\times e_h^{\bf E}\rangle_{L^2(K)^3}-\langle e_{\pi}^{\bf E}, \nabla\times e_h^{\bf H}\rangle_{L^2(K)^3} \Big)\\
&+\sum_{F\in{\mathcal G}_h^{int}}\Big(
\langle \beta_Ke_{\pi,K_F}^{\bf H}+\beta_{K_F}e^{\bf H}_{\pi,K}-\gamma_F {\bf n}_F\times [[e^{\bf E}_{\pi}]]_{F}, {\bf n}_F\times [[e_h^{\bf E}]]_{F} \rangle_{L^2(F)^3}\\
&-\langle \alpha_K e^{\bf E}_{\pi,K_F}+\alpha_{K_F}e^{\bf E}_{\pi,K}+\delta_F {\bf n}_F\times [[e_{\pi}^{\bf H}]]_{F}, {\bf n}_F\times [[e_h^{\bf H}]]_{F} \rangle_{L^2(F)^3}
\Big)\\
&+\sum_{F\in{\mathcal G}_h^{ext}} \langle e_{\pi}^{\bf H}, {\bf n}_F\times e_h^{\bf E}\rangle_{L^2(F)^3}-2\gamma_{F}
\langle {\bf n}_{F}\times e_{\pi}^{\bf E}, {\bf n}_F\times e_h^{\bf E}\rangle_{L^2(F)^3},
\end{align*}
where $e_{\pi}=\left((e_{\pi}^{\bf E})^{\top},\; (e_{\pi}^{\bf H})^{\top}\right)^{\top}$ and $e_{h}=\left((e_{h}^{\bf E})^{\top},\; (e_{h}^{\bf H})^{\top}\right)^{\top}$.
The property of the projection $\pi_h$ leads to $ \langle e_{\pi}^{\bf H}, \nabla\times e_h^{\bf E}\rangle_{L^2(K)^3}=\langle e_{\pi}^{\bf E}, \nabla\times e_h^{\bf H}\rangle_{L^2(K)^3}=0$.
Then using Cauchy-Schwarz and Young's inequalities, we have
\begin{align}\label{eq 4.15}
&\langle M_h e_{\pi}(s), e_h(s)\rangle_{\mathbb V}
\leq \sum_{F\in{\mathcal G}_h^{ext}}\gamma_{F} \|{\bf n}_{F}\times e^{\bf E}_h\|^2_{L^2(F)^3}
\nonumber\\
& +
\sum_{F\in{\mathcal G}_h^{int}}\Big(\frac
{\gamma_{F}}{2}\|{\bf n}_{F}\times[[e^{\bf E}_h]]_F\|^2_{L^2(F)^3}+\frac{\delta_{F}}{2}\|{\bf n}_{F}\times[[e^{\bf H}_h]]_F\|^2_{L^2(F)^3}\Big) 
\nonumber\\
&+\sum_{F\in{\mathcal G}_h^{int}}\Big(\frac{1}{2\gamma_{F}}\|\beta_Ke_{\pi,K_F}^{\bf H}+\beta_{K_F}e^{\bf H}_{\pi,K}-\gamma_F {\bf n}_F\times [[e^{\bf E}_{\pi}]]_{F}\|_{L^2(F)^3}^2\nonumber\\
&\qquad\qquad+\frac{1}{2\delta_{F}}\|\alpha_K e^{\bf E}_{\pi,K_F}+\alpha_{K_F}e^{\bf E}_{\pi,K}+\delta_F {\bf n}_F\times [[e_{\pi}^{\bf H}]]_{F}\|_{L^2(F)^3}^2
\Big)\\
&+\sum_{F\in{\mathcal G}_h^{ext}}\Big( \frac{1}{2\gamma_F}\|e_{\pi}^{\bf H}\|^2_{L^2(F)^3} + 2\gamma_{F}\|{\bf n}_{F}\times e_{\pi}^{\bf E}\|^2_{L^2(F)^3} \Big)\nonumber\\
\leq &-\frac12  \langle M_he_h(s), e_h(s)
\rangle_{\mathbb V} +Ch^{2k-1}|u(s)|_{H^{k}(D)^6}^{2},\nonumber
\end{align}
where in the last step, we use the equality in (ii) of Proposition \ref{properties of discrete dG} and the inequality \eqref{projection order boundary}.
Hence, we have
$$
\frac12 \|e_h(t)\|^2_{\mathbb V} -\frac12\int_0^t \langle M_he_h(s), e_h(s)
\rangle_{\mathbb V}{\rm d}s \leq Ch^{2k-1}\int_0^t|u(s)|_{H^k(D)^6}^2{\rm d}s.
$$
 Proposition \ref{properties of discrete dG} (ii) yields that the second term on the left-hand side is nonnegative.
Then taking expectation and using Lemmas \ref{exact_H1} and \ref{exact_H2}, we get $\sup_{t\in[0,T]}{\mathbb E}\|e_h(t)\|^2_{\mathbb V}\leq Ch^{2k-1}\int_0^T{\mathbb E}|u(s)|_{H^k(D)^6}^2{\rm d}s,$
which combines with \eqref{eq 4.14} completes the proof.
\end{proof}

\section{Full discretization of stochastic Maxwell equations}\label{sec:5}
In this section, we consider the full discretization of stochastic Maxwell equations \eqref{sto_evo} by applying the midpoint scheme in time and the dG method with the upwind fluxes in space:
\begin{equation}\label{full discretization}
u_h^{n+1}=u_h^{n}+\frac{\tau}{2}\left( M_h u_h^{n}+M_h u_h^{n+1} \right)-\pi_h\Delta W^{n+1},
\end{equation}
with $u_h^0=\pi_h u_0$. Utilizing the basis of ${\mathbb V}_h$ in Section \ref{sec4}, the fully discrete method \eqref{full discretization} can be rewritten as the midpoint scheme for \eqref{SODE},
\[
A{\bf u}^{n+1}=A{\bf u}^{n}+\frac{\tau}{2}\big(B{\bf u}^{n}+B{\bf u}^{n+1}\big)-\Delta {\bf W}^{n+1}.
\]

Following the proof of Proposition \ref{prop: divergence free}, the divergence conservation property \eqref{sto_max3} is preserved numerically by the solution of \eqref{full discretization}  in a weak sense.
 \begin{proposition}\label{prop: divergence free full}
 Let $u_0\in{\mathbb V}_0$ and $Q^{\frac12}\in HS({\mathbb V},{\mathbb V}_0)$. The solution $\left\{u_h^n\right\}_{0\leq n\leq N}$ of the fully discrete method \eqref{full discretization} satisfies:  $\forall ~n\in\{0,1,\ldots,N\}$, and $\forall ~\phi\in X_h$, 
 \[
 \langle \nabla\cdot(\varepsilon {\bf E}_h^n), \phi\rangle_{-1}=
  \langle \nabla\cdot(\mu {\bf H}_h^n), \phi\rangle_{-1}=0,\qquad \qquad {\mathbb P}\mbox{-a.s.}
 \]
 \end{proposition}
 \begin{proof}
 For $\psi,\phi\in X_h$, using the definition of the inner product $\langle\cdot, \cdot\rangle_{-1}$, we get
 \begin{align*}
\left\langle\begin{pmatrix}\nabla\cdot(\varepsilon{\bf E}_h^{n+1})\\ \nabla\cdot(\mu{\bf H}_h^{n+1})\end{pmatrix}, \begin{pmatrix} \psi \\ \phi\end{pmatrix}\right\rangle_{-1}
&=\langle \nabla\cdot(\varepsilon{\bf E}_h^{n+1}),  \psi\rangle_{-1}
+\langle \nabla\cdot(\varepsilon{\bf H}_h^{n+1}),  \phi\rangle_{-1}\\
&=-\left\langle\begin{pmatrix}{\bf E}_h^{n+1}\\ {\bf H}_h^{n+1}\end{pmatrix}, \begin{pmatrix} \nabla\psi \\ \nabla\phi\end{pmatrix}\right\rangle_{\mathbb V}.
 \end{align*}
 Using \eqref{full discretization} we obtain
 \begin{align*}
 \left\langle\begin{pmatrix}{\bf E}_h^{n+1}\\ {\bf H}_h^{n+1}\end{pmatrix}, \begin{pmatrix} \nabla\psi \\ \nabla\phi\end{pmatrix}\right\rangle_{\mathbb V}
 =&\left\langle\begin{pmatrix}{\bf E}_h^n\\ {\bf H}_h^n\end{pmatrix}, \begin{pmatrix} \nabla\psi \\ \nabla\phi\end{pmatrix}\right\rangle_{\mathbb V}
 +\frac{\tau}{2} \left\langle M_h\begin{pmatrix}{\bf E}_h^n+{\bf E}_h^{n+1}\\ {\bf H}_h^n+{\bf H}_h^{n+1}\end{pmatrix}, \begin{pmatrix} \nabla\psi \\ \nabla\phi\end{pmatrix}\right\rangle_{\mathbb V}\\
& -\left\langle\pi_h\begin{pmatrix} \varepsilon^{-1}\Delta W_e^{n+1}\\  \mu^{-1}\Delta W_m^{n+1}\end{pmatrix}, \begin{pmatrix} \nabla\psi \\ \nabla\phi\end{pmatrix}\right\rangle_{\mathbb V}.
 \end{align*}
 Using Proposition \ref{properties of discrete dG} (iii),  the second term on the right-hand side equals to zero, since for any function $\varphi \in X_h$, we have $\nabla\times\nabla \varphi={\bf 0}$, ${\bf n}_{F}\times [[\nabla \varphi]]_{F}={\bf 0}$ for $F\in {\mathcal G}_h^{\rm int}$ and ${\bf n}\times \nabla\varphi={\bf 0}$ on $\partial D$. 
 For the third term  on the right-hand side, the property  of the projection \eqref{property of projection}, and the fact that $ \begin{pmatrix} \nabla\psi \\ \nabla\phi\end{pmatrix}\in {\mathbb V}_h$  yield
 \begin{align*}
 \left\langle\pi_h\begin{pmatrix} \varepsilon^{-1}\Delta W_e^{n+1}\\  \mu^{-1}\Delta W_m^{n+1}\end{pmatrix}, \begin{pmatrix} \nabla\psi \\ \nabla\phi\end{pmatrix}\right\rangle_{\mathbb V}
& = \left\langle\begin{pmatrix} \varepsilon^{-1}\Delta W_e^{n+1}\\ \mu^{-1}\Delta W_m^{n+1}\end{pmatrix}, \begin{pmatrix} \nabla\psi \\ \nabla\phi\end{pmatrix}\right\rangle_{\mathbb V}\\
& =-\langle \nabla\cdot (\Delta W_e^{n+1}), \psi\rangle_{-1}-\langle \nabla\cdot (\Delta W_m^{n+1}),\phi\rangle_{-1}=0.
  \end{align*}
   Thus, 
   \[
   \left\langle\begin{pmatrix}{\bf E}_h^{n+1}\\ {\bf H}_h^{n+1}\end{pmatrix}, \begin{pmatrix} \nabla\psi \\ \nabla\phi\end{pmatrix}\right\rangle_{\mathbb V}
 =\left\langle\begin{pmatrix}{\bf E}_h^n\\ {\bf H}_h^n\end{pmatrix}, \begin{pmatrix} \nabla\psi \\ \nabla\phi\end{pmatrix}\right\rangle_{\mathbb V} =\cdots =\left\langle\begin{pmatrix}{\bf E}_h^0\\ {\bf H}_h^0\end{pmatrix}, \begin{pmatrix} \nabla\psi \\ \nabla\phi\end{pmatrix}\right\rangle_{\mathbb V}=0,
   \]
   where in the last step, we use
    \[
 \left\langle\begin{pmatrix}{\bf E}_h^0\\ {\bf H}_h^0\end{pmatrix}, \begin{pmatrix} \nabla\psi \\ \nabla\phi\end{pmatrix}\right\rangle_{\mathbb V}
 =\left\langle \pi_h\begin{pmatrix}{\bf E}_0\\ {\bf H}_0\end{pmatrix}, \begin{pmatrix} \nabla\psi \\ \nabla\phi\end{pmatrix}\right\rangle_{\mathbb V}
 =\left\langle \begin{pmatrix}{\bf E}_0\\ {\bf H}_0\end{pmatrix}, \begin{pmatrix} \nabla\psi \\ \nabla\phi\end{pmatrix}\right\rangle_{\mathbb V}=0.
 \]
 Therefore  the conclusion of this proposition comes from taking  $\phi=0$ or $\psi=0$, respectively.
 \end{proof}
 
The mild version of the full discretization  \eqref{full discretization} can be  rewritten as
\begin{equation}\label{full discretization mild}
u_n^{n+1}=S_{h,\tau} u_h^n- T_{h,\tau} \pi_h\Delta W^{n+1},
\end{equation}
where $T_{h,\tau}=\big(I-\frac{\tau}{2}M_h\big)^{-1}$ and $S_{h,\tau}=\big(I-\frac{\tau}{2}M_h\big)^{-1}\big(I+\frac{\tau}{2}M_h\big)$.

\begin{lemma}\label{lemma 5.1}
For operators $T_{h,\tau}$ and $S_{h,\tau}$  on ${\mathbb V}_h$, the following estimates hold:
\begin{itemize}
\item[(i)] $\|T_{h,\tau}\|_{{\mathcal L}({\mathbb V}_h,{\mathbb V}_h)}\leq 1$.
\item[(ii)] $\|S^n_{h,\tau}\|_{{\mathcal L}({\mathbb V}_h,{\mathbb V}_h)}\leq 1$ for any $0\leq n\leq N$.
\end{itemize}
\end{lemma}
\begin{proof}
To prove the assertion (i), we define $\widetilde{v}=T_{h,\tau}v$ for any $v\in {\mathbb V}_h$, which means that $\widetilde{v}=v+\frac{\tau}{2}M_h\widetilde{v}$. Taking the inner product with $\widetilde{v}$ yields
$
\frac12\Big[\|\widetilde{v}\|_{\mathbb V}^2-\|v\|_{\mathbb V}^2+\|\widetilde{v}-v\|_{\mathbb V}^2 \Big]=\frac{\tau}{2}\langle M_h\widetilde{v}, \widetilde{v}\rangle_{\mathbb V} \leq 0.
$
Hence $\|\widetilde{v}\|_{\mathbb V}=\|T_{h,\tau}v\|_{\mathbb V}\leq \|v\|_{\mathbb V}$ leads to the assertion (i).

Similarly, to prove the assertion (ii), we define $v_h^{n}=S_{h,\tau}^nv$ for any $v\in {\mathbb V}_h$, which means that
$
v_h^{\ell}=v_h^{\ell-1}+\frac{\tau}{2}\big(M_hv_h^{\ell-1}+M_hv_h^{\ell} \big), ~ \ell=1,2,\ldots, n,
$
with $v_h^0=v$. Taking the inner product with $(v_h^{\ell-1}+v_h^{\ell})$ yields
$
\|v_h^{\ell}\|_{\mathbb V}^2-\|v_h^{\ell-1}\|_{\mathbb V}^2
 \leq 0,
$
and thus $\|v_h^{\ell}\|_{\mathbb V}\leq \|v_h^{\ell-1}\|_{\mathbb V}\leq \ldots \leq \|v_h^{0}\|_{\mathbb V}=\|v\|_{\mathbb V}$. This leads to the assertion (ii).
\end{proof}

\begin{proposition}
There exists a constant $C$ independent of $h$ and $\tau$ such that
\[\max_{0\leq n\leq N}{\mathbb E}\|u_h^n\|^2_{\mathbb V}\leq C(1+{\mathbb E}\|u_0\|^2_{\mathbb V}).\]
\end{proposition}
\begin{proof}
From \eqref{full discretization mild}, we know that
$
u_h^n=S_{h,\tau}^n\pi_hu_0-\sum_{j=1}^{n}S^{n-j}_{h,\tau}T_{h,\tau}\pi_h\Delta W^j.
$
Taking  $\|\cdot\|_{\mathbb V}$-norm on both sides of the above equation and using the triangle inequality, we get
\begin{align*}
{\mathbb E}\|u_h^n\|_{\mathbb V}^2&\leq 2{\mathbb E}\|S_{h,\tau}^n\pi_hu_0\|_{\mathbb V}^2+2{\mathbb E}\Big\|\sum_{j=1}^{n}S^{n-j}_{h,\tau}T_{h,\tau}\pi_h\Delta W^j\Big\|_{\mathbb V}^2\\
&\leq 2{\mathbb E}\|\pi_hu_0\|_{\mathbb V}^2+2\sum_{j=1}^{n}{\mathbb E}\left\|\pi_h\Delta W^j\right\|_{\mathbb V}^2\leq 2{\mathbb E}\|u_0\|_{\mathbb V}^2+2T{\rm Tr}(Q),
\end{align*}
which completes the proof.
\end{proof}

Let $W_{M;N,h}:=\sum\limits_{j=1}^{N}S_{h,\tau}^{N-j}T_{h,\tau}\pi_h\Delta W^{j}$. Then it is Gaussian on ${\mathbb V}$ with mean $0$ and covariance operator 
$$
Q_{T;N,h}:={\rm Cov}(W_{M;N,h})=\tau\sum_{j=1}^{N}\big(S_{h,\tau}^{N-j}T_{h,\tau}\pi_h\big)Q(S_{h,\tau}^{N-j}T_{h,\tau}\pi_h\big)^{*}.
$$

Applying the fully discrete method to the small noise system \eqref{sto_evo_para}, we denote by $\{{\mathcal L}\big( u_h^{N;u_0,\lambda} \big)\}_{\lambda>0}$ the laws of the fully discretizations.
The asymptotic behavior of $\{{\mathcal L}\big( u_h^{N;u_0,\lambda} \big)\}_{\lambda>0}$ is similar to that of $\{{\mathcal L}\big( u^{N;u_0,\lambda} \big)\}_{\lambda>0}$ in Proposition \ref{LDP_semi}, which is stated below.
\begin{proposition}\label{LDP_full}
For integer $N>0$ and $u_0\in{\mathbb V}$, the family of distributions $\{{\mathcal L}\big( u_h^{N;\;u_0,\lambda} \big)\}_{\lambda >0}$ satisfies the  large deviation principle with the good rate function
\begin{equation}
I_{T;N,h}^{u_0}(v)=\begin{cases}
\frac12\| \big(Q_{T;N,h}\big)^{-\frac12}\big(v-S_{h,\tau}^N\pi_h u_0\big) \|_{\mathbb V}^2, & v-S_{h,\tau}^N\pi_h u_0\in \big(Q_{T;N,h}\big)^{\frac12}({\mathbb V}),\\
+\infty,& {\rm otherwise}.
\end{cases}
\end{equation}
\end{proposition}

\subsection{Error estimate of full discretization}
The error $u_h^n-u(t_n)$ is divided as
$
u_h^n-u(t_n)=\left(u_h^n-u^n\right)+\left( u^n-u(t_n)\right),
$
where the second term in the right-hand side is the error in temporal direction, which has been studied in Proposition \ref{error of midpoint}.
Hence we only need to consider the error $u_h^n-u^n$. By inserting the term $\pi_h u^n$, we get 
$
u_h^n-u^n=\left(u_h^n-\pi_h u^n\right)+\left(\pi_h u^n-u^n\right)=:e_h^{n}+e_{\pi}^{n}.
$

Note that  \eqref{projection order} and Propositions \ref{H1 of un}-\ref{H2 of un} yield that, for $k\in\{1,2\}$, 
\[
\left({\mathbb E}\|e_{\pi}^n\|_{\mathbb V}^2\right)^{\frac12}\leq  Ch^k \left( {\mathbb E}\|u^n\|_{H^k(D)^6}^2\right)^{\frac12}\leq Ch^k \left(1+ {\mathbb E}\|u_0\|_{{\mathcal D}(M_0^k)}^2\right)^{\frac12}.
\]
The estimate of error $e_h^n$ is stated in the following theorem.
\begin{theorem}\label{thm 5.1}
Let $\{u^n,~0\leq n\leq N\}$ in $L^2(\Omega;  H^k(D))$ with $k\in\{1,2\}$ be the solution of \eqref{midpoint} and let $\{u_h^n,~0\leq n\leq N\}$ in $L^2(\Omega;{\mathbb V}_h)$ be the solution of \eqref{full discretization}. Then there is a constant $C$ independent of $h$ and $\tau$ such that

\begin{equation}\label{eq 5.2}
\max_{0\leq n\leq N}\Big({\mathbb E}\|e_h^n\|^2_{\mathbb V}\Big)^{\frac12}\leq Ch^{k-\frac12} \quad \mbox{for } k\in\{1,2\}.
\end{equation}
\end{theorem}

\begin{proof}
We apply the projection $\pi_h$ to the temporal semidiscretization 
\eqref{midpoint} and use Proposition \ref{properties of discrete dG} (i) to get
\begin{equation} \label{projection un}
\pi_h u^{n+1}=\pi_h u^n +\frac{\tau}{2} \left( M_h u^n+M_h u^{n+1} \right)-\pi_h\Delta W^{n+1}.
\end{equation}
Subtracting \eqref{full discretization} from \eqref{projection un} yields,
\begin{align}
e_n^{n+1}=e_h^n+\frac{\tau}{2}\left(M_he_h^n +M_h e_h^{n+1}\right)+\frac{\tau}{2}\left(M_he_{\pi}^n+ M_h e_{\pi}^{n+1}\right).
\end{align}
Applying $\langle\cdot, e_h^n+e_h^{n+1} \rangle_{\mathbb V}$, we obtain
\begin{align}\label{error}
\|e_h^{n+1}\|_{\mathbb V}^2-\|e_h^{n}\|_{\mathbb V}^2
=\frac{\tau}{2}\langle M_h(e_h^n + e_h^{n+1}), e_h^n+e_h^{n+1} \rangle_{\mathbb V}+\frac{\tau}{2}\langle M_h(e_{\pi}^n+ e_{\pi}^{n+1}), e_h^n+e_h^{n+1} \rangle_{\mathbb V}.
\end{align} 
For the second term on the right-hand side of \eqref{error}, we use \eqref{eq 4.15} to get
\begin{align*}
\langle M_h(e_{\pi}^n+ e_{\pi}^{n+1}), e_h^n+e_h^{n+1} \rangle_{\mathbb V}\leq& -\frac12\langle M_h(e_h^n + e_h^{n+1}), e_h^n+e_h^{n+1} \rangle_{\mathbb V}\\
&+Ch^{2k-1}\|u^{n}+u^{n+1}\|^2_{H^k(D)^6}.
\end{align*}
Hence \eqref{error} becomes
\begin{align*}
\|e_h^{n+1}\|_{\mathbb V}^2-\|e_h^{n}\|_{\mathbb V}^2
\leq \frac{\tau}{4}\langle M_h\left(e_h^n + e_h^{n+1}\right), e_h^n+e_h^{n+1} \rangle_{\mathbb V}+C\tau h^{2k-1}\|u^{n}+u^{n+1}\|^2_{H^k(D)^6}.
\end{align*}
 Proposition \ref{properties of discrete dG} (ii) leads to 
$
\langle M_he_h^n +M_h e_h^{n+1}, e_h^n+e_h^{n+1} \rangle_{\mathbb V}\leq 0,
$
and then
\[
{\mathbb E}\|e_h^{n+1}\|_{\mathbb V}^2-{\mathbb E}\|e_h^{n}\|_{\mathbb V}^2
\leq C\tau h^{2k-1}{\mathbb E}\left(\|u^{n}\|^2_{H^k(D)^6}+\|u^{n+1}\|^2_{H^k(D)^6}\right)\leq C\tau h^{2k-1}.
\]
Gronwall's inequality yields the conclusion. 
\end{proof}

Combining the error estimates in temporal and spatial directions, we finally obtain the error estimate for the full discretization \eqref{full discretization}.
\begin{theorem}
If the assumptions of Theorem \ref{thm 5.1} and Theorem \ref{error of midpoint} are satisfied, then the fully discrete error $u_h^n-u(t_n)$ is bounded by
\[
\left({\mathbb E}\|u_h^n-u(t_n)\|_{\mathbb V}^2\right)^{\frac12}\leq C\tau^{\frac{k}{2}}+Ch^{k-\frac12},\quad \mbox{for }k\in\{1,2\},
\]
where the constant $C$ is independent of $h$ and $\tau$.
\end{theorem}

\bibliographystyle{plain}
\bibliography{maxwell}
\end{document}